\documentclass{article}

\usepackage[accepted]{icml2019blank} 

\usepackage{natbib, algorithm, algorithmic,authblk,hyperref}

\usepackage{microtype}
\usepackage{graphicx}
\usepackage{subcaption}
\usepackage{booktabs} 

\usepackage{multirow}

\usepackage{amsfonts}
\usepackage{amsmath}
\usepackage{amsthm}
\usepackage{amssymb} 

\usepackage{nicefrac}
\usepackage{mathtools}

\usepackage{mdframed} 
\usepackage{amsthm}

\usepackage{mdframed} 
\usepackage{thmtools}

\definecolor{shadecolor}{gray}{0.95}
\declaretheoremstyle[
headfont=\normalfont\bfseries,
notefont=\mdseries, notebraces={(}{)},
bodyfont=\normalfont,
postheadspace=0.5em,
spaceabove=1pt,
mdframed={
  skipabove=8pt,
  skipbelow=8pt,
  hidealllines=true,
  backgroundcolor={shadecolor},
  innerleftmargin=4pt,
  innerrightmargin=4pt}
]{shaded}

\declaretheorem[style=shaded,within=section]{definition}
\declaretheorem[style=shaded,sibling=definition]{theorem}

\declaretheorem[style=shaded,sibling=definition]{assumption}
\declaretheorem[style=shaded,sibling=definition]{corollary}

\declaretheorem[style=shaded,sibling=definition]{lemma}

 \usepackage{xcolor}
\usepackage{color}
\usepackage{graphicx}
\graphicspath{{./figures/}}  

\usepackage{hyperref}

\newcommand{\R}{\mathbb{R}} 

\newcommand{\cB}{{\cal B}}

\newcommand{\cD}{{\cal D}}

\newcommand{\cF}{{\cal F}}

\newcommand{\cQ}{{\cal Q}}

\newcommand{\bI}{{\bf I}}
\newcommand{\bS}{{\bf S}}

\newcommand{\mA}{{\bf A}}
\newcommand{\mB}{{\bf B}}

\newcommand{\mD}{{\bf D}}

\newcommand{\mG}{{\bf G}}
\newcommand{\mH}{{\bf H}}
\newcommand{\mI}{{\bf I}}

\newcommand{\mR}{{\bf R}}
\newcommand{\mS}{{\bf S}}

\newcommand{\mU}{{\bf U}}
\newcommand{\mV}{{\bf V}}
\newcommand{\mW}{{\bf W}}

\newcommand{\mZ}{{\bf Z}}

\usepackage[colorinlistoftodos,bordercolor=orange,backgroundcolor=orange!20,linecolor=orange,textsize=scriptsize]{todonotes}

\newcommand{\eqdef}{\coloneqq} 

\newcommand{\dotprod}[1]{\left< #1\right>} 
\newcommand{\norm}[1]{ \left\| #1 \right\|}      

\newcommand{\Prob}[1]{\mathbb{P}[#1]}

\newcommand{\Diag}[1]{\mathbf{Diag}\left( #1\right)}

\providecommand{\trace}[1]{{\rm Trace}\left( #1\right)}

\newcommand{\E}[1]{\mathbb{E}\left[#1\right] } 
\newcommand{\EE}[2]{\mathbb{E}_{#1}\left[#2\right] }

\newcommand{\sym}{\mathbb{S}}
\newcommand{\prob}[1]{\mathbb{P}\left(#1\right)}
\newcommand{\normfro}[2]{\norm{#1}_{F(#2)}}

\icmltitlerunning{Fast Linear Convergence of Randomized BFGS}
\usepackage[utf8]{inputenc}

\title{\bf Fast Linear Convergence of Randomized BFGS}

\author[1]{Dmitry Kovalev}
\author[2]{Robert M.\ Gower}
\author[1]{Peter Richt\'arik}
\author[1,3,4]{Alexander Rogozin} 
\affil[1]{King Abdullah University of Science and Technology (KAUST), Thuwal, Saudi Arabia}
\affil[2]{
T\'el\'ecom Paris, IPP, Paris, France}
\affil[3]{Moscow Institute of Physics and Technology, Dolgoprudny, Russia}
\affil[4]{Sirius University of Science and Technology, Sochi, Russia\footnote{This work was conducted while A. Rogozin was a research intern in the Optimization and Machine Learning Lab of Peter Richt\'arik at KAUST; this visit was funded by the KAUST Baseline Research Funding Scheme. The research of A. Rogozin was also partially supported by RFBR, project number 19-31-51001.}}

\date{February 9, 2020\\
(Revised: \today)}

\begin{document}
\maketitle

\begin{abstract}
 Since the late 1950's when  quasi-Newton methods first appeared, they have become one of the most widely used and efficient algorithmic paradigms for unconstrained optimization. Despite their immense practical success, there is little theory that shows why these methods are so efficient. We provide a local rate of convergence for a randomized BFGS method which can be significantly better than that of gradient descent, thus giving theoretical evidence supporting the superior empirical performance of quasi-Newton methods.
\end{abstract}

{\footnotesize
\tableofcontents
}
\newpage
\section{Introduction}

In this paper we consider the optimization problem
\begin{equation}\label{eq:prob}
x_* \in \arg\min\limits_{x \in \R^d} f(x),
\end{equation}
where $f: \R^d \rightarrow \R$ is a twice continuously differentiable function. While the BFGS method~\citep{broyden1967,fletcher1970,goldfarb1970,shanno1970conditioning} is one of the most efficient and celebrated algorithms for solving~\eqref{eq:prob}, a clear theoretical justification for its success, or the success of {\em any} quasi-Newton method, has been elusive. We do know, however, that the quasi-Newton methods converge $Q$--superlinearly~\citep{Powell1971}, and this is often pointed to as the justification for their success. Yet this superlinear convergence occurs only asymptotically, in an arbitrarily small ball around the solution, and at an unknown superlinear rate. In particular, this superlinear rate could be arbitrary close to linear.

In this work we provide {\em the first meaningful convergence rate of  (a randomized variant of) BFGS}. By meaningful, we mean a rate that can be faster then the rate  of gradient descent,  thus giving much stronger support to the practical success of the BFGS method than what has been available so far.  Furthermore, for our results to hold, we only need to assume  $f$ to be self-concordant. 

Let $\sym^d $ denote the set of $d \times d$ symmetric  matrices, and let $\mH_{x} \eqdef \nabla ^2 f(x)$ be the Hessian matrix of $f$ evaluated at $x$. For $x=x_k$ we will further abbreviate $\mH_{k}\eqdef \mH_{x_k}$.  In this paper we consider the randomized BFGS update first introduced by~\citet{inverse}, and later used in the context of machine learning by~\citet{GowerGold2016}, given by the formula
\begin{equation} 
\text{BFGS}(\mB,\mH, \mS)  \eqdef  \mG 	+ \left(\mI - \mG \mH \right)\mB \left(\mI - \mH \mG \right),
		\label{eq:BFGS}
\end{equation}
where 
$\mG = \mG(\mH, \mS)\eqdef \mS(\mS^\top  \mH \mS)^{-1}\mS^\top,$
 $\mS\in \R^{d \times \tau} $ is a (typically thin) random matrix with $\tau \ll d$ columns and full column rank, $\mH \in \sym^d$ is a non-singular target Hessian matrix and $\mB \in \sym^d$ is the current estimate of the inverse Hessian matrix. Note that $\mS^\top  \mH \mS$ is invertible since we  assume that $\mS$ has full column rank and $\mH$ is invertible. We use this BFGS update in  the randomized BFGS method stated as Algorithm~\ref{alg:BFGS}.

\begin{algorithm}
	\caption{RBFGS}
	\label{alg:BFGS}
	\begin{algorithmic}[1]
		\STATE {\bf Parameters:} $x_0 \in \R^d, \mB_0 \in \sym^{d}$, distribution $\cD$ over matrices from $\R^{d\times \tau}$ 
		\FOR{$k = 0,1,2, \ldots$}
			\STATE $x_{k+1} = x_k - \mB_k\nabla f(x_k)$ \label{ln:step}

			\STATE   {\bf Monotonic option}:
			\STATE  $\quad x_{k+1} = \arg\min \left\{f(x_{k+1}), f(x_k) \right\}$  
			 \label{ln:monotonic} 
			\STATE Sample random $\mS_k \sim \cD$
			\STATE Compute 
			$\mB_{k+1}
			=
			\text{BFGS}(\mB_k,\mH_k,\mS_k)
			$						 
		\ENDFOR
	\end{algorithmic}
\end{algorithm}

In order to update our estimate of the inverse Hessian $\mH_k^{-1}$, we use a {\em random linear} measurement $\mH_k \mS_k$ of the true Hessian $\mH_k$. Random linear transformations are often alternatively called {\em sketches}, and we adopt this terminology.  We thus refer to the random matrix $\mS_k \in \R^{d\times \tau}$  by the name {\em  sketching matrix,} and to the  product $\mH_k \mS_k$ by the name {\em Hessian sketch}. Note that the sketch can be computed by doing $\tau$ directional derivatives of the gradient since
\begin{equation}\label{eq:HS}
 \left. \frac{d}{dt}\nabla f(x_k+ t s_i )\right|_{t=0}  \;= \; \nabla ^2 f(x_k) s_i\;=\; \mH_k s_i, 
\end{equation}
for $i=1,\ldots, \tau,$ where $s_i$ is the $i$th column of $\mS_k.$ 

Standard BFGS performs very similar iterations to RBFGS, but uses {\em deterministic} sketching matrices of the form $\mS_k=x_k-x_{k-1}$ (in particular, BFGS is forced to use $\tau=1$), where  $x_k,x_{k-1}\in \R^d$ are the last two iterates produced by BFGS, and  {\em approximates} the Hessian sketch via first-order information as follows: $\mH_k \mS_k \approx \nabla f(x_k) - \nabla f(x_{k-1})$. 

Our forthcoming theoretical results hold for any distribution $\cD $ of sketching matrices.  Based on the intuition our general theory provides,  in the case of generalized linear models (see Section~\ref{sec:GLM})  we also develop a particular distribution $\cD$  based on the singular value decomposition.
In Section~\ref{sec:numerics} we show that our new sketching matrix is well suited to highly ill-conditioned problems. We also compare our new sketch  against  other standard sketches and the original BFGS method in a series of numerical experiments.

\subsection{Background}

We now recall the background on the convergence of BFGS, some modern randomized BFGS variants and their convergence results, and also techniques surrounding self-concordance.

\paragraph{Classic convergence results.} The BFGS method was first shown to converge locally and superlinearly about 50 years ago~\citep{Powell1971}. This proof was later extended to include a larger family of quasi-Newton methods by
\citet{BroydenDennisMore1973}.  Much more recently, \citet{GoldfarbGao2016b} showed that the {\em block} BFGS method first introduced by \cite{GowerGold2016} also converges superlinearly (and without the need for line search).

\paragraph{Modern randomized variants of BFGS.}
\citet{Byrd2015}  proposed the SQN method which uses  a single Hessian-vector product in the BFGS update, as opposed to using differences of stochastic gradients.
Subsequently,~\citet{Moritz2015} proposed combining  SQN with the variance reduced gradient method SVRG~\citep{Johnson2013}, and provided a global linear convergence rate for the resulting method.
\citet{GowerGold2016} extended the SQN method to allow for sketching matrices~\eqref{eq:HS}, and also provide an improved linear convergence rate. However, their rate is still orders of magnitude worse than the rate of convergence of gradient descent. 
  Very recently, \citet{SchmidtLacost2019} gave a global convergence rate for a stochastic variant of the BFGS method combined with stochastic gradient descent.
The reason that the rates of convergence in all of these previous works are significantly inferior to that of gradient descent, is that they do not factor in the contribution of the quasi-Newton matrix towards the convergence. Instead, the previous analysis focus on obtaining bounds on the estimated inverse Hessian 
\begin{equation}\label{eq:spectrabnd}
 c \mI  \preceq \mB_k \preceq C \mI, \quad \mbox{where} \quad 0<c<C,
\end{equation}
and then use this to ``bound away'' the contribution of the quasi-Newton method, for which a cost is paid. Following this step, the analysis follows verbatim  the standard analysis of gradient-based methods, albeit with the added burden that the constants $c$ and $C$~\eqref{eq:spectrabnd} bring.

\paragraph{Self-concordance.} 
It was a revolution\footnote{A precursor of this result made the front page of 1979 New York Times entitled ``Soviet Discovery Rocks World of Mathematics''. It also made the headlines of the Guardian, with a humorously incorrect title ``Soviet Answer to Travelling Salesman Problem''~\citep{EncyclopediaOpt}.}  in the optimization community when, in the late 70’s and early 80's, it was shown that a large class of convex optimization problems\footnote{The self-concordant barrier functions with closed bounded domains.} could be solved in polynomial time by interior point methods~\citep{nesterov1994interior}.
A key concept that facilitated this revolution was that of self-concordance, which describes a large class of convex functions whose second order derivative could be naturally ``controlled'' by the third derivative. We  rely on self-concordant functions in this paper in Theorem~\ref{theo:sc}, and in the proof, show how the convergence of the BFGS method can leverage self-concordance.

\subsection{Notation and definitions}

Both quasi-Newton methods and self-concordance are defined using weighted norms.  For every $\mW \in \R^{d\times d}$ and $\mH \in \sym^d$ and positive definite we write
\[\normfro{\mW}{\mH}^2 \eqdef  \trace{\mH \mW \mH \mW^\top} = \norm{\mH^{1/2}\mW\mH^{1/2}}_F^2\] to denote a weighted Frobenius norm.  Let 
$$\norm{v}_x^2 \eqdef \dotprod{\mH_x v,  v}$$ 
be the \emph{local inner product}~\citep{Renegar:2001}. To further abbreviate our notation, we will use $\norm{v}_* \eqdef \norm{v}_{x_*}$.
Let 
\[\cB^\delta_x \;\eqdef \; \{y\in \R^d \, : \, \norm{y-x}_x < \delta  \}\] 
be the ball of size $\delta>0$ around $x$ under the local norm.  Using the local inner product, we can now state the definition of the key concept of self-concordance.
\begin{definition}\label{def:sc} A functional $f:\R^d\to \R$ is self-concordant if for all $x$ in its domain the  Hessian is positive definite, and for each $y\in \cB_x^1$ and $ v \neq 0$ we have 
\begin{equation}\label{eq:selfcon}
1-\norm{y-x}_x \leq   \frac{\norm{v}_y}{ \norm{v}_x} \leq \frac{1}{1-\norm{y-x}_x}.
\end{equation}
\end{definition}

 \section{Convergence Results}
 
 In this section we establish our main convergence results: local linear convergence for self-concordant functions and for smooth and strongly convex functions, and superlinear convergence.

\subsection{Local linear convergence for self-concordant functions}

Under the assumption of self-concordance only, we now prove local linear convergence of the randomized BFGS method in the following theorem.
\begin{theorem} \label{theo:sc}
Let
	\begin{equation}\label{eq:rho2} \rho \eqdef \inf\limits_{x \in \R^d} \lambda_{\min}\left(\E{\mH_{x}^{1/2}\mS(\mS^\top  \mH_x \mS)^{-1}\mS^\top \mH_{x}^{1/2}}\right)
	\end{equation}
and consider the Lyapunov function
$$
\Phi_{\sigma}^k  \eqdef  \sigma\norm{\mB_{k} - \mH_*^{-1}}_{F(\mH_*)}^2 +   \norm{x_k -x_*}_*,
$$
where $\sigma \eqdef \frac{3}{ \rho}$.
If $f$ is self-concordant and 
\begin{equation}\label{eq:convradap}\Phi_{\sigma}^0 \leq  \frac{1}{2}\min \left \{ \frac{3}{2} - \frac{1}{2}\sqrt{1+ 8\sqrt{\frac{1-\rho}{1-\frac{2}{3}\rho}}},  \;\;   \rho\frac{2-\rho}{69  d  +5 \rho}  
 \right\} ,
\end{equation}
%
then Algorithm~\ref{alg:BFGS} converges linearly according to
\begin{equation}\label{eq:Phirecur}
\E{\Phi_{\sigma}^{k+1}  }\; \leq \; \left( 1- \frac{\rho}{2}\right)  \E{\Phi_{\sigma}^{k}}.
\end{equation}
Unrolling this recurrence, we get
$$
\E{\norm{x_k -x_*}_*} \; \leq \;  \left( 1- \frac{\rho}{2}\right)^k \Phi_{\sigma}^0 .
$$
\end{theorem}

We now provide a few key insights and a brief outline of the proof technique of Theorem~\ref{theo:sc}; the  complete proof is given in Appendix~\ref{sec:prooftheosc}. Our starting point is the work of \citet{inverse} who  studied randomized algorithms for inverting a fixed invertible\footnote{An extension to the computation of the Moore-Penrose pseudoinverse of a general rectangular matrix was developed in \citep{PSEUDOINVERSE}.} matrix $\mH$. Their methods can be interpreted as  randomized, non-adaptive (and possibly block) variants of classical quasi-Newton matrix inversion formulas.  In particular, they studied randomized BFGS updates for inverting $\mH\succ 0$ and showed that the sequence of random matrices given by  the randomized BFGS rule $$\mB_{k+1} = \text{BFGS}(\mB_k, \mH, \mS_k)$$  converges linearly to $\mH^{-1}$ in mean square. 

In this work, we face the additional challenge of  a moving target $\mH$. Indeed, in our setting, $\mH$ is not fixed, but is set to $\mH = \mH_k = \nabla^2 f(x_k)$ in iteration $k$, and hence changes from iteration to iteration. If the Hessian $\mH_k$ changes too fast, then there is no reason why the randomized BFGS update should be able to ``catch up'' with the change, and be any good at maintaining a good estimate of the inverse of  $\mH_k$ throughout the iterations, let alone improve the estimate and converge to $(\nabla^2 f(x_*))^{-1}$ as $k\to \infty$.

This is the place where self-concordance comes to our aid.
Indeed,  self-concordance provides  control over how fast the Hessian, and the inverse Hessian, changes as the point $x_k$ changes; see Lemma~\ref{lem:hesscontrol} in the Appendix. By carefully combining these intuitions, we are able to control how much $\mB_k$ deviates from the Hessian $\mH_k$; see Lemma~\ref{lem:Bkplus1contract12}.  Yet stringing this  result into an induction argument is difficult due to the changing local norm. Fortunately, self-concordance allows us to ``change the norm'' to another point~\eqref{eq:selfcon}. Using this defining property of self-concordance, we state all of our results with respect to the local norm at the optimal point. This allows for a more straightforward induction argument and the resulting recurrence in~\eqref{eq:Phirecur}.

\subsection{Local linear convergence for smooth and strongly convex functions}\label{sec:b98fg9809f}

Here we prove that BFGS converges with same rate of convergence 
as in Theorem~\ref{theo:sc}, but under an alternative (though  also common) set of assumptions.
We also now use the Monotonic option on line~8 of Algorithm~\ref{alg:BFGS}. 
This enforces that the function values $f(x_k)$ are monotonically decreasing. That is, let
$\cQ \eqdef \{x \; : \; f(x) \leq f(x_0)\}.$
By~\eqref{ln:monotonic} we have that $x_k \in \cQ$ for all $k$.

Instead of self-concordance, we now rely on strong convexity, Lipschitz gradients and Hessians.
\begin{assumption} \label{ass:strconv}
	Function $f$ is $\mu$-strongly convex. That is, 
	\begin{equation}\label{eq:strconv}
		f(y) \geq f(x) + \dotprod{\nabla f(x), y-x} + \frac{\mu}{2}\norm{y-x}_2^2, \quad \forall x,y \in \R^d.
	\end{equation}
\end{assumption}

\begin{assumption} \label{ass:GradLip}
	Function $f$ has $L_1$-Lipschitz gradient. That is,
	\begin{equation} \label{eq:GradLip}
		\norm{\nabla f(y) - \nabla f(x)}_2 \leq L_1 \norm{y-x}_2, \quad \forall x,y \in \R^d.
	\end{equation}
\end{assumption}

\begin{assumption}\label{ass:HessLip}
	Function $f$ has $L_2$-Lipschitz Hessian. That is,
	\begin{equation} \label{eq:HessLip}
		\norm{\mH_{y} - \mH_{x}}_2 \leq L_2 \norm{y-x}_2, \quad \forall x,y \in \R^d.
	\end{equation}
\end{assumption}

Under these assumptions we have the following theorem.
\begin{theorem} \label{theo:locallinear}

Let Assumptions~\ref{ass:strconv},~\ref{ass:GradLip} and~\ref{ass:HessLip} hold and let
$$
\Psi_k = \sqrt{f(x_k) - f(x_*)} + \beta \norm{\mB_k - \mH_*^{-1}}_{F(\mH_*)}^2,
$$
where  $\beta = \frac{4\sqrt{2}L_1^{5/2}}{\mu L_2 \rho}$ and where $\rho$ is defined in~\eqref{eq:rho2}.
	If $f(x_0) - f(x_*) \leq \cF$, where 
	\begin{equation}\label{eq:1}
		\cF \eqdef \frac{1}{4}\left[
		\frac{\sqrt{2L_1}L_2}{\mu^2}  + \frac{32\sqrt{2}dL_1^{5/2}L_2}{\rho\mu^4}
		\right]^{-2},
	\end{equation}
then
	$
		\E{\Psi_{k+1} } \leq \left(1 - \frac{ \rho }{2} \right)\E{\Psi_k},
	$
	where 
 	\begin{equation}\label{eq:rho} \rho \eqdef \inf\limits_{x \in \cQ} \lambda_{\min}\left(\E{\mH_{x}^{1/2}\mS(\mS^\top  \mH_x \mS)^{-1}\mS^\top \mH_{x}^{1/2}}\right).\end{equation}
\end{theorem}

A disadvantage of using Assumptions~\ref{ass:strconv},  \ref{ass:GradLip},  and \ref{ass:HessLip}, as we do in Theorem~\ref{theo:locallinear},  as compared to self-concordance,  as we do in Theorem~\ref{theo:sc}, is that now the region of local convergence~\eqref{eq:1} depends on constants which can be hard, or almost numerically impossible (such as $\mu$), to compute. As such, it would be very difficult to use Theorem~\ref{theo:locallinear} together with globalization strategies to develop a global linear convergence. In contrast, we only need to have a bound on $\rho$ to compute the size of the region of convergence in  Theorem~\ref{theo:locallinear} (see \eqref{eq:convradap}). We give some insights into bounding $\rho$ in Section~\ref{sec:exe}. The downside of Theorem~\ref{theo:locallinear} is that the norms $\norm{\cdot}_*$ cannot be computed since it relies on $x_*.$ This issue can be dealt with by using a continuation scheme to design a globalization strategy \citep{Renegar:2001}.

\subsection{Superlinear convergence}

For completion, we also prove the superlinear convergence of Algorithm~\ref{alg:BFGS} with high probability. 
\begin{theorem}\label{thm:superlinear}
	$\sqrt{f(x_k) - f(x_*)}$ converges to 0 at superlinear rate with probability 1.
\end{theorem}
\begin{proof}
The proof follows by combining Lemma~\ref{lem:1} and Lemma~\ref{lem:2}.
\end{proof}

\section{Examples and Applications}
\label{sec:exe}
To gain some insight into the  consequences of Theorems~\ref{theo:locallinear} and~\ref{theo:sc},  we first examine two extreme cases for the choice of the sketching matrix: 
\begin{itemize}
\item[(i)] $\mS$ is invertible (full curvature),
\item[(ii)]$\mS$ has only one column.
\end{itemize} 
We then develop a more involved SVD sketch, and apply it to generalized linear models and linear programming.

\subsection{Invertible $\mS$} 
When the sketching matrix $\mS$ is invertible with probability one, then $\rho =1$. In view of~\eqref{eq:convradap}, if 
\begin{align*}
3\norm{\mB_{0} - \mH_*^{-1}}_{F(\mH_*)}^2 +   \norm{x_0 -x_*}_*  &\leq   \min \left \{\frac{1}{4},   \frac{1}{2}\frac{1}{69  d  +5 }\right\}\\
& =  \frac{1}{2}\frac{1}{69  d  +5 }\,,
\end{align*}
then 
$\norm{x_k -x_*}_*  \leq   \left( \frac{1}{2}\right)^k \Phi_{\sigma}^0 .$
This is not surprising. Indeed, given that when $\mS$ is invertible, it is not hard to show that $\mB_k = \mH_k^{-1}$; that is,  the quasi-Newton matrix equals the inverse Hessian. Consequently, Algorithm~\ref{alg:BFGS} is equivalent to Newton's method.

\subsection{One column $\mS$}
On the other extreme, if we choose the sketch matrix to have a single column only, $\mS =s \in \R^d$, then $\rho$ can be much smaller than $1$ and will also depend on the spectrum of the Hessian.
\begin{corollary}[Single column sketches]\label{cor:singlevecsketch}
Let $0 \prec \mU \in \R^{n\times n}$ be a symmetric positive definite matrix such that $\mH_x \preceq \mU, \; \forall x \in \R^d$.
Let $\mD =[d_1,\ldots, d_n] \in \R^{n\times n}$ be
a given invertible matrix such that $d_i^\top \mH_x d_i \neq 0$ for all $x \in \R^d$ and $i =1,\ldots, n$. Furthermore, suppose that $f$ is convex and the Hessian matrix is Lipschitz. If we sample according to $\prob{\mS_k = d_i} = p_i \eqdef \frac{d_i^\top \mU d_i}{\trace{\mD^\top \mU \mD}},$ then,  under the assumptions of Theorem~\ref{theo:locallinear}, Algorithm~\ref{alg:BFGS} converges at a rate of at least
\begin{equation}\label{eq:rhosinglesketch}
  \rho \; \geq \; \min_{x \in \cQ}\frac{\lambda_{\min}^+(\mH_x^{1/2} \mD \mD^\top \mH_x^{1/2})}{\trace{\mD^\top \mU\mD}}.
\end{equation}
\end{corollary}
\begin{proof} This quantity $\rho$ also appears as the convergence rate of the randomized subspace Newton method of \cite{RSN}. In particular, 
it was shown in~Corollary~1 in~\citep{RSN} that under these assumptions, \eqref{eq:rhosinglesketch} holds.
\end{proof}
Thus, by carefully choosing the sketching matrix $\mS$, we can see that the rate of convergence $\rho$ is bounded below as in~\eqref{eq:rhosinglesketch}. Since in this example we have not assumed the spectrum of the Hessian to be bounded from below,  $\rho$ still depends on $\mH_x$ in~\eqref{eq:rhosinglesketch}.

\subsection{Generalized linear models}
\label{sec:GLM}

We now show that in situations where the Hessian matrix has a bounded spectrum, we can choose a sketching matrix so that $\rho$ essentially does \emph{not depend} on the spectra or ill-conditioning of the Hessian matrix.  This leads to a more precise bound on the rate of convergence $\rho$ that highlights how it can be arbitrarily bigger (leading to faster rate) than the rate of convergence of gradient descent. This also extends to accelerated gradient descent.

\begin{definition}
Let $0 \leq \ell \leq u.$ Let $\phi_i: \R \mapsto \R_+$ be a twice differentiable function such that
\begin{equation} \label{eq:phipripribnd}
  \phi_i''(t) \; \in [\ell, u], \quad \mbox{and} \quad \phi_i'''(t) \leq U\quad \mbox{for} \quad i =1,\ldots, n.
\end{equation}
Let $a_i \in \R^{d}$ for $i=1,
\ldots, n$ and $\mA = [a_1,\ldots, a_n] \in \R^{d\times n}.$ We say that $f:\R^d\to \R$ is a generalized linear model when
\begin{equation} \label{eq:fgenlin}
\textstyle  f(x) = \frac{1}{n} \sum \limits_{i=1}^n \phi_i (\dotprod{a_i, x})\;.
\end{equation}
\end{definition}

\begin{theorem}\label{prop:svd_sketch} Let $\mA$ have full row rank. Let $f$ be given by~\eqref{eq:fgenlin}. Let $\mA  = \mU \Sigma \mV^\top$ be the reduced singular value decomposition of $\mA$ with $\mU \in \R^{d\times d}, \Sigma \in \R^{d \times d}$ and $\mV \in \R^{n \times d}.$
Consider a sketching matrix such that 
\begin{equation}
\Prob{\mS = \mU \Sigma^{-1} e_i}  =  \frac{1}{d}, \quad  \mbox{for} \quad  i =1,\ldots, n. \label{eq:svdsketch}
\end{equation}
It follows that $\rho$ give in~\eqref{eq:rho2} is bounded by
\begin{equation}\label{eq:rhobndGLM}
1- \frac{\rho}{2} \leq   1-  \frac{1}{2}\frac{\ell}{u} \frac{1}{d}.
\end{equation}
\end{theorem}

\begin{proof}
 Since $\mA$ has full row rank we have that $\Sigma$ is invertible, $\mU$ is orthonormal and $\mV^\top \mV = \mI.$ Consequently $\mA^\dagger = \mV \Sigma^{-1} \mU^\top $. The Hessian of $f$  is thus given by
\begin{equation}
\mH_x  =   \frac{1}{n} 
\mA \Phi'' (\mA^\top x) \mA^\top\;,\label{eq:genlinhess2}
\end{equation}
where $\Phi'' (\mA^\top x) = \Diag{\phi''_1(\dotprod{a_1, x}), \ldots ,\phi''_n(\dotprod{a_n, x})}.$  It is now not hard to show that 
\begin{equation}
 \frac{\ell \sigma_{\min}(\mA)^2}{n} \mI  \preceq  \mH_x   \preceq   \frac{u \sigma_{\max}(\mA)^2}{n} \mI,
\end{equation}
and thus $f$ is strongly convex with Lipschitz gradient. Furthermore, $\phi'''_i(t) \leq U$ guarantees that the Hessian is Lipschitz.
Consequently, Theorem~\ref{theo:locallinear} holds. It now remains to bound $\rho$. With probability $\frac{1}{d}$, we have
 \begin{align} 
\mS^\top \mH_x \mS &=   \frac{1}{n} e_i^\top  \Sigma^{-1} \mU^\top \mU \Sigma \mV^\top \Phi'' (\mA^\top x)   \mV \Sigma \mU^\top \mU \Sigma^{-1}  e_i  \nonumber \\
 &= \frac{1}{n} e_i^\top \mV^\top  \Phi'' (\mA^\top x) \mV e_i.\label{eq:s99ks94s}
 \end{align}
 Now, since $\mV^\top \mV =\mI$, we get
$
 \ell \mI \preceq  \mV^\top  \Phi'' (\mA^\top x) \mV \preceq  u \mI.
$ Consequently,
\begin{eqnarray*}
\rho  &=&  \lambda_{\min}\left(\E{\mH_{x}^{1/2}\mS(\mS^\top  \mH_x \mS)^{-1}\mS^\top \mH_{x}^{1/2}} \right)\\
 &=& \lambda_{\min}\left(\E{\frac{\mH_{x}^{1/2}\mS \mS^\top \mH_{x}^{1/2}}{\frac{1}{n} e_i^\top \mV^\top  \Phi'' (\mA^\top x) \mV e_i} }\right) \\
&\geq & \frac{n}{u}\lambda_{\min}\left(\mH_{x}^{1/2}\E{\mS \mS^\top } \mH_{x}^{1/2}\right) \\
&= & \frac{n}{u} \lambda_{\min}\left(\mH_{x}^{1/2}\mU \Sigma^{-1}\E{e_i e_i^\top} \Sigma^{-1}\mU^\top \mH_{x}^{1/2}\right) \\
&= & \frac{n}{u d} \lambda_{\min}\left( \Sigma^{-1} \mU^\top \mH_{x}\mU \Sigma^{-1}\right) \\
& = & \frac{1}{u d} \lambda_{\min}\left( \mV^\top  \Phi'' (\mA^\top x) \mV\right)  \\
&\geq & \frac{\ell}{u} \frac{1}{d}.
\end{eqnarray*}

\end{proof}

We refer to the sketch in~\eqref{eq:svdsketch} as the SVD-sketch.
To contrast the bound on $\rho$ given in~\eqref{eq:rhobndGLM}, let us compare it to the rate of convergence of gradient descent. 
Gradient descent converges at a rate of 
\begin{equation}
 \rho_{GD} \eqdef 1 - \frac{\max_x \lambda_{\min}(\mH_x)}{\min_x \lambda_{\max}(\mH_x)} \nonumber\\
 \leq  1 -\frac{\ell}{u} \frac{\sigma_{\min}(\mA)^2}{\sigma_{\max}(\mA)^2},\label{eq:rate}
\end{equation}
and is thus at the mercy of the condition number 
\[\kappa \eqdef \frac{\sigma_{\max}(\mA)}{\sigma_{\min}(\mA)}\] 
of the matrix $\mA.$ To highlight this, consider the extreme case where
$\mA$ is a square $d \times d$ Hilbert matrix~\citep{david1894ein}. In this case  \[\kappa =O\left((1+\sqrt{2})^{4d} \right), \] and thus  $\rho_{GD}$ grows exponentially with $d$ while $\rho$  grows at most linearly due to~\eqref{eq:rhobndGLM} .  Even the rate of convergence of accelerated gradient descent~\citep{Nesterov1998} would grow exponentially with an exponent of $2d$ in this case. 

Though in practice it is well known that the convergence of gradient descent is far more sensitive to ill conditioning as compared to the convergence of the BFGS method, to the best of our knowledge, this is the first clear theoretical justification to support this observation.

\subsection{Linear programs}

The original problem that motivated the introduction of self-concordant functions was linear programming,
\begin{align}
\min_{x \in \R^d}& \dotprod{c,x} \quad \mbox{subject to } \quad \mA x \leq b,  \label{eq:lp}
\end{align}
where $c \in \R^d, \mA \in \R^{n \times d}$ and $b\in \R^n$ are the given data. 
Interior point methods and barrier methods for solving~\eqref{eq:lp} are based on solving a sequence of problems such as
\begin{align}
\textstyle f(x) = \lambda \dotprod{c,x} - \sum \limits_{i=1}^n \log(b_i - \dotprod{a_i,x}), \label{eq:logbarrier}
\end{align}
where $ \lambda >0$ is the barrier parameter, see~\citep{Renegar:2001}.
The objective function in~\eqref{eq:logbarrier} is self-concordant.  Furthermore, the logarithmic barrier enforces that the constraint $\mA x < b$ hold. The Hessian of~\eqref{eq:logbarrier} is given by
\begin{align}\label{eq:hesslogbarrier}
 \nabla^2 f(x) = \mA^\top \Diag{\frac{1}{(b_i - \dotprod{a_i,x})^2}}\mA.
\end{align}

Let $x^k$ be the iterates of Algorithm~\ref{alg:BFGS} applied to minimizing~\eqref{eq:logbarrier} with the monotonic option on line~8.
\begin{theorem}\label{prop:svd_sketch_log} Let $\mA$ have full row rank, $f$ be given by~\eqref{eq:logbarrier}, and $\mA  = \mU \Sigma \mV^\top$ be the reduced singular value decomposition of $\mA$ with $\mU \in \R^{d\times d}, \Sigma \in \R^{d \times d}$ and $\mV \in \R^{n \times d}.$
Consider the SVD-sketch given in~\eqref{eq:svdsketch}.
Let
\begin{align}
u  = \frac{1}{\displaystyle \min_{i=1,\ldots, n}\; \min_{x\; : \; \mA x <b } (b_i - \dotprod{a_i,x})^2 } ,\nonumber\\
\ell  = \frac{1}{\displaystyle \max_{i=1,\ldots, n}\; \max_{x\; : \; \mA x < b } (b_i - \dotprod{a_i,x})^2 }.\label{eq:ellulog}
\end{align}
It follows that $\rho$ given in~\eqref{eq:rho2} is bounded by
\begin{equation}\label{eq:rhobndipm}
1- \frac{\rho}{2} \leq   1-  \frac{1}{2}\frac{\ell}{u} \frac{1}{d}.
\end{equation}
\end{theorem}

As the barrier parameter $ \lambda$ becomes larger,  the minimizer of~\eqref{eq:logbarrier} will get closer to the boundary of the set $\mA x \leq b.$ Consequently, the Hessian~\eqref{eq:hesslogbarrier} becomes increasing ill-conditioned, which is a well known issue with log-barrier methods. This in turn also implies that the constants $u$ and $\ell$ in~\eqref{eq:ellulog} will increase and decrease, respectively, as $ \lambda$ grows. Thus like all methods for solving the log-barrier problem~\eqref{eq:logbarrier}, this ill-conditioning needs to be treated with care.

\section{Numerical Experiments}
\label{sec:numerics}
\newcommand{\figmult}{0.45} 
\graphicspath{{./figures/}}

\graphicspath{{./figures/}}

We run numerical tests for the Randomized BFGS algorithm and compare three different distributions of $\mS$:
\begin{enumerate}
	\item \textbf{gauss}: All entries of $\mS$ are generated independently and have standard Gaussian distribution.
	\item\label{item:svd_sketch} \textbf{svd}: SVD-sketch from Proposition \ref{prop:svd_sketch}, i.e., for singular value decomposition of $\mA  = \mU \Sigma \mV^\top$ take $\Prob{\mS = \mU \Sigma^{-1} e_i} = \tfrac{1}{d} \quad  \mbox{for  } i =1,\ldots, n.$
	\item\label{item:coord} \textbf{coord}: Coordinate sketch, i.e., $\Prob{\mS = e_i} = \tfrac{1}{d}$.
\end{enumerate}
We also experiment with the sketch size, i.e. the number of columns $\tau$ of the sketching matrix $\mS$. We use \textbf{distr}\_$\tau$ in the legends of our plots to denote which distribution (\textbf{distr}) and which sketch size $\tau$ was used. For example, \textbf{gauss}\_$10$ corresponds to using a Gaussian sketch with sketch size $10$.
Finally, we compare the performance of RBFGS to the classical BFGS and accelerated Nesterov method both on synthetic and real data. For a fair comparison across algorithms, we use the Wolfe line search~\citep{Wolfe1971} to determine the stepsizes for the Randomized BFGS, classical BFGS and Nesterov methods.

\subsection{Synthetic Quadratic Problem}\label{subsection:exper_synthetic}

\begin{figure}
	\centering
	\begin{subfigure}[t]{0.23\textwidth}
		\centering
		\includegraphics[width = \textwidth]{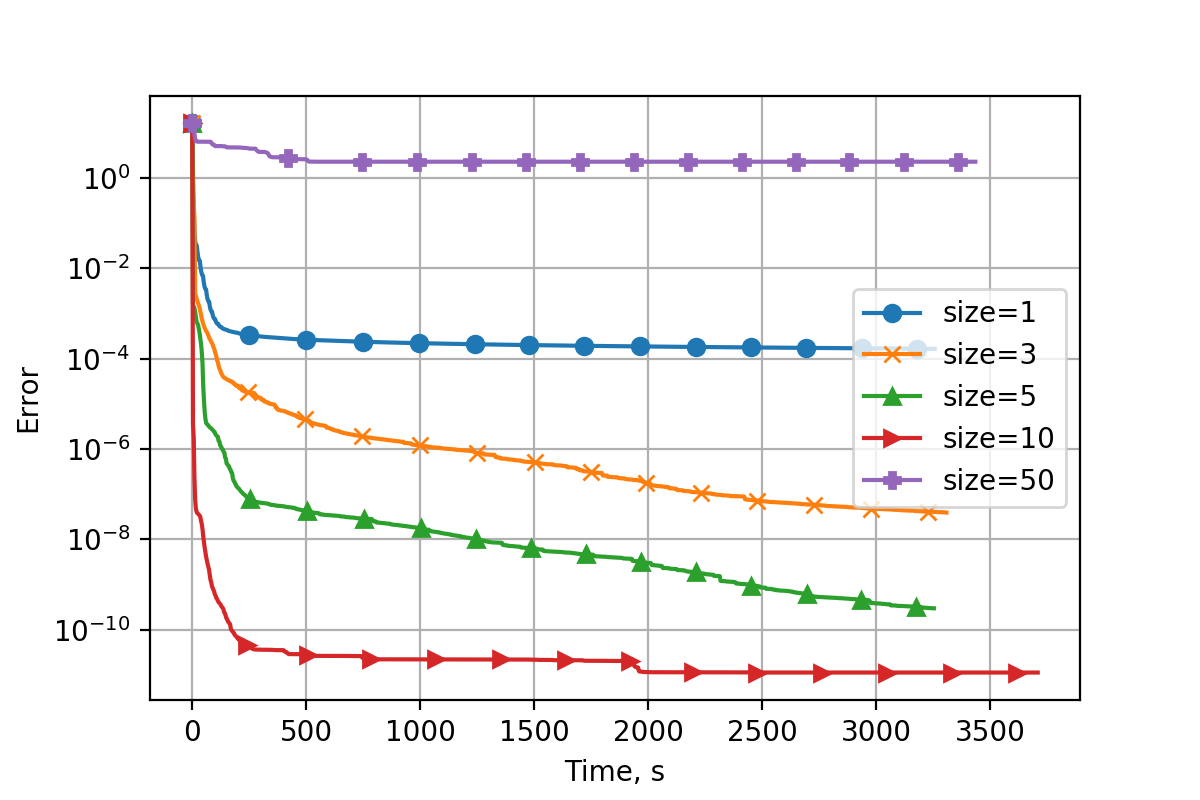}
		\caption{gauss}
	\end{subfigure}
	\begin{subfigure}[t]{0.23\textwidth}
		\centering
		\includegraphics[width = \textwidth]{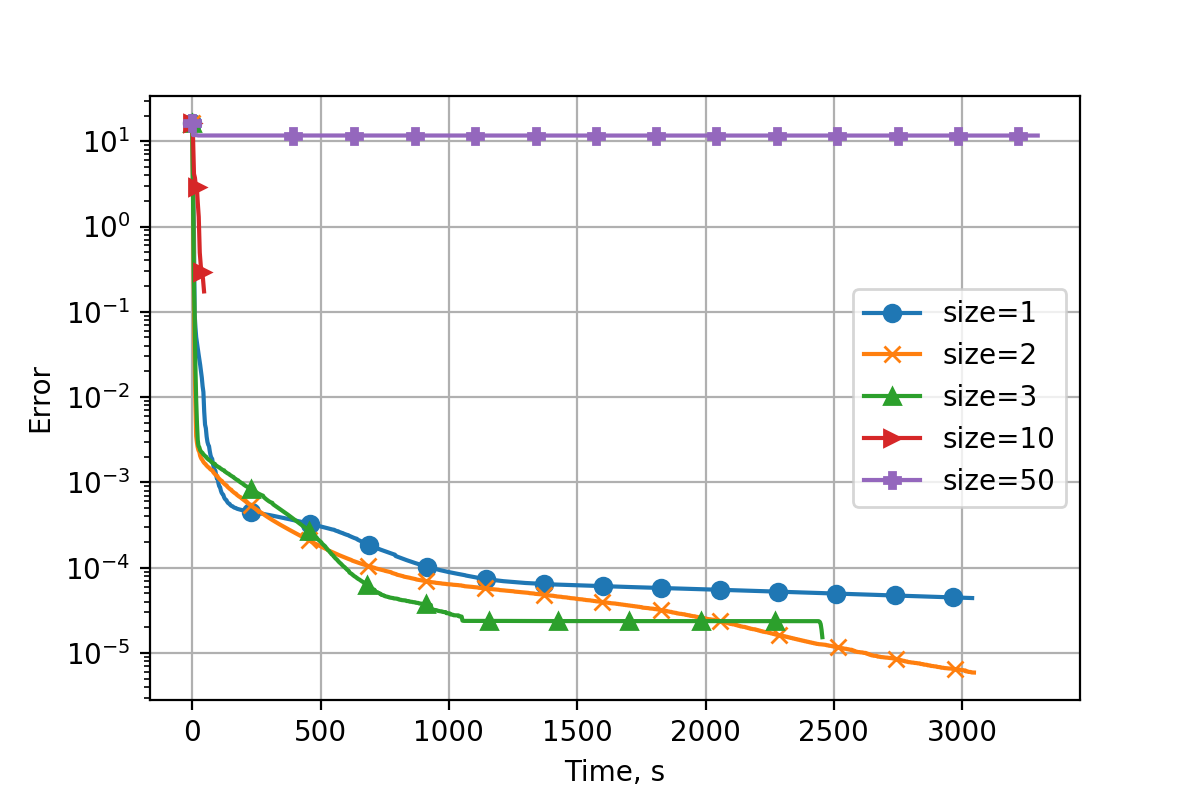}
		\caption{coord}
	\end{subfigure}
	\\
	\begin{subfigure}[t]{0.23\textwidth}
		\centering
		\includegraphics[width = \textwidth]{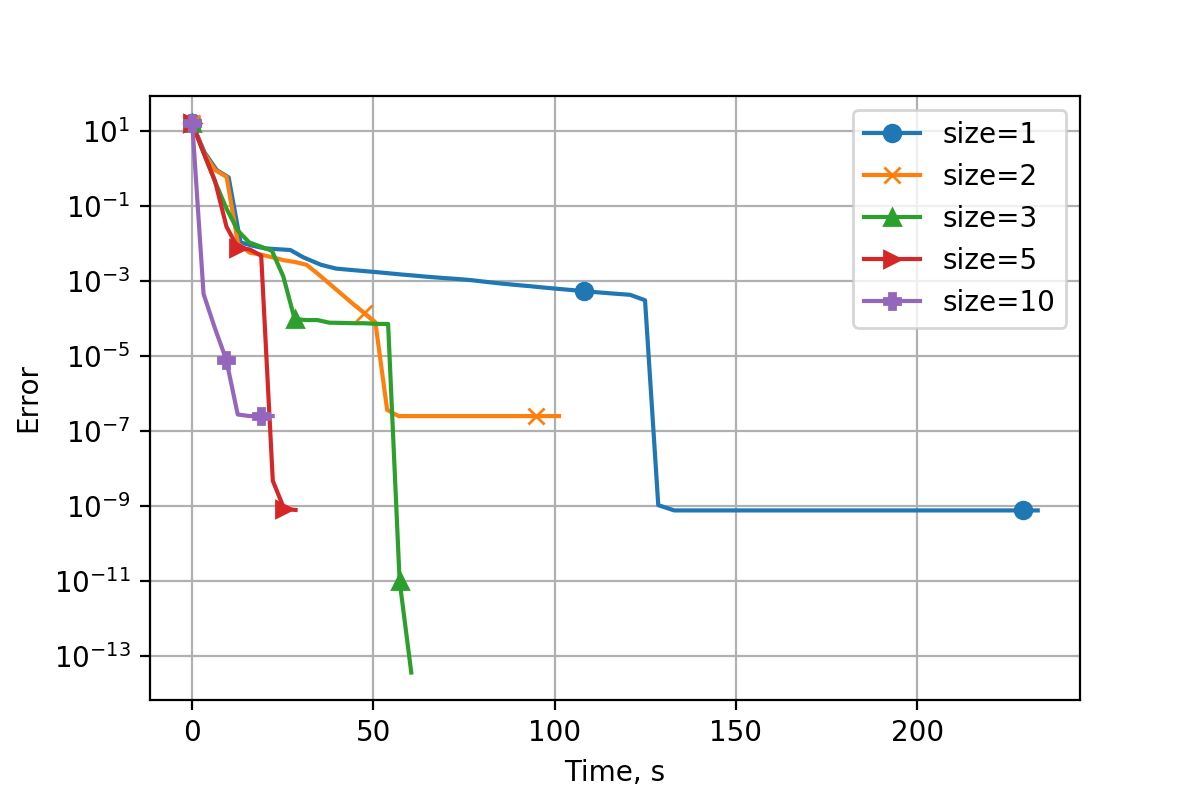}
		\caption{svd}
	\end{subfigure}
	\begin{subfigure}[t]{0.23\textwidth}
		\centering
		\includegraphics[width = \textwidth]{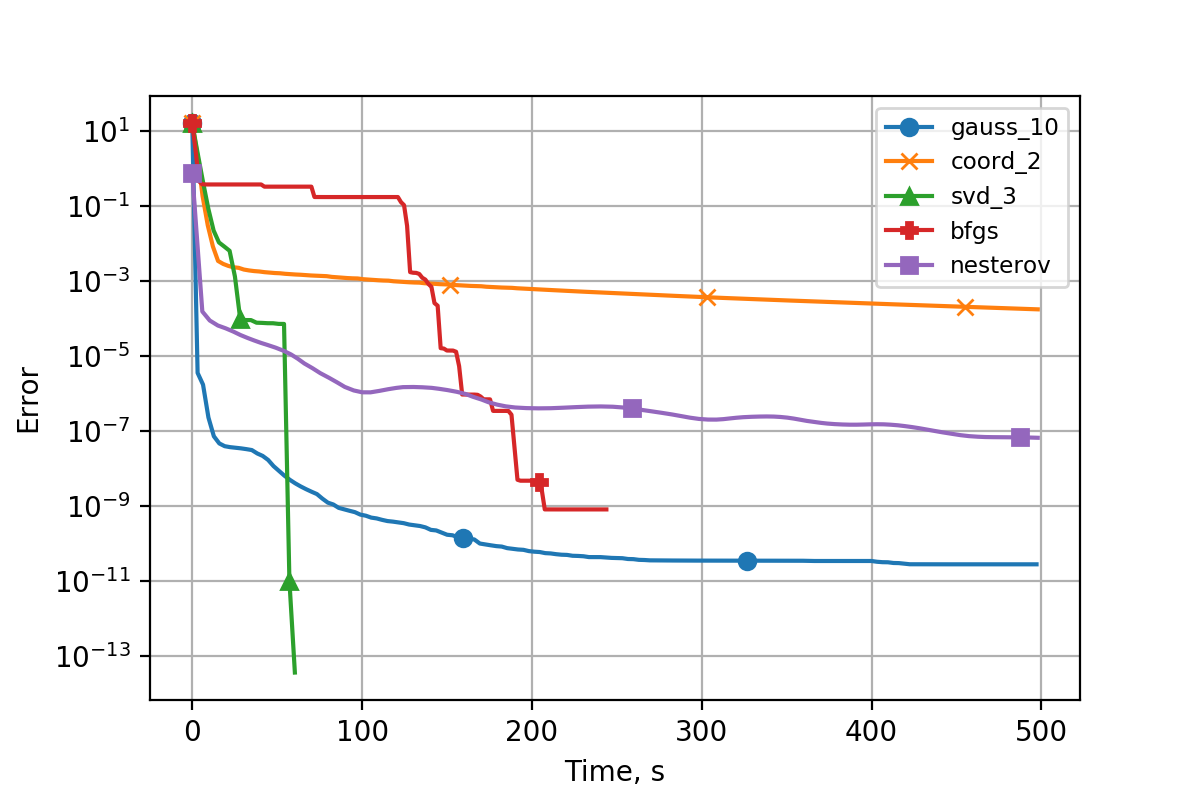}
		\caption{methods compared}
	\end{subfigure}
	\caption{Hilbert matrix $10000\times 10000$}
	\label{fig:Hilbert}
\end{figure}

Consider a quadratic problem
\begin{equation*}
\min\limits_{x \in \R^d} \left\{f(x) \eqdef \tfrac{1}{2} \norm{\mA x}_2^2 =\tfrac{1}{2} x^\top (\mA^\top \mA) x \right\}, 
\end{equation*}
where $\mA\in\R^{d\times d}$ is a Hilbert matrix~\citep{david1894ein} defined as $[\mA]_{ij} = \tfrac{1}{i + j - 1}$ for $  i = 1,\ldots,n.$ This problem is interesting due to the large condition number of $\mA$. 
For each distribution, we experiment with different sketch sizes and choose the one that gives the best results in terms of computational time. Another possibility is to compare the number of iterations. However, due to the fact that the complexity of  Algorithm~\ref{alg:BFGS} depends on $\tau$, this would not lead to a fair comparison.

When using the \textbf{svd} sketch, we only include singular values above tolerance $10^{-8}$. This is done to improve numerical stability. As a benchmark, we also compare to the classical BFGS method and accelerated Nesterov gradient descent, see Figure~\ref{fig:Hilbert}. We observe in Figure~\ref{fig:Hilbert} that for this particular quadratic problem small sketch sizes ($10$ or below) perform better. We note that \textbf{svd} distribution yielded the best results. This is to be expected, since according to Proposition~\ref{prop:svd_sketch}, the rate of convergence when using the svd sketch is essentially invariant to ill-conditioning.

\subsection{Classification Problem on Real Data}

Now consider the logistic regression problem with L2 regularizer given by
\begin{equation}
\textstyle	f(x) = \frac{1}{n}\sum\limits_{i=1}^n \log\left(1 + \exp(-b_i\dotprod{a_i, x})\right) + \frac{\lambda}{2}\norm{x}_2^2,
\end{equation}
where $a_1,\ldots,a_n\in\R^d$ are the data points, $b_1,\ldots,b_n\in\{-1, 1\}$ are the class labels and $\lambda> 0$ is a regularization coefficient. In our experiments, we used LIBSVM datasets \cite{Chang2011}. We set the regularization parameter $\lambda\sim 10^{-3} L$, where $L$ denotes the smoothness constant of logistic loss without regularizer. Similar to Section \ref{subsection:exper_synthetic}, we experiment with different sketch sizes and choose ones which perform better. 
  
The experiments show that a good general strategy of choosing sketch size is $\tau\sim\sqrt{d}$; see Figures~\ref{fig:w8a},~\ref{fig:a9a},~\ref{fig:covtype},~\ref{fig:gisette} and ~\ref{fig:colon_cancer}.

\begin{figure}[t]
	\centering
	\begin{subfigure}[t]{0.23\textwidth}
		\centering
		\includegraphics[width = \textwidth]{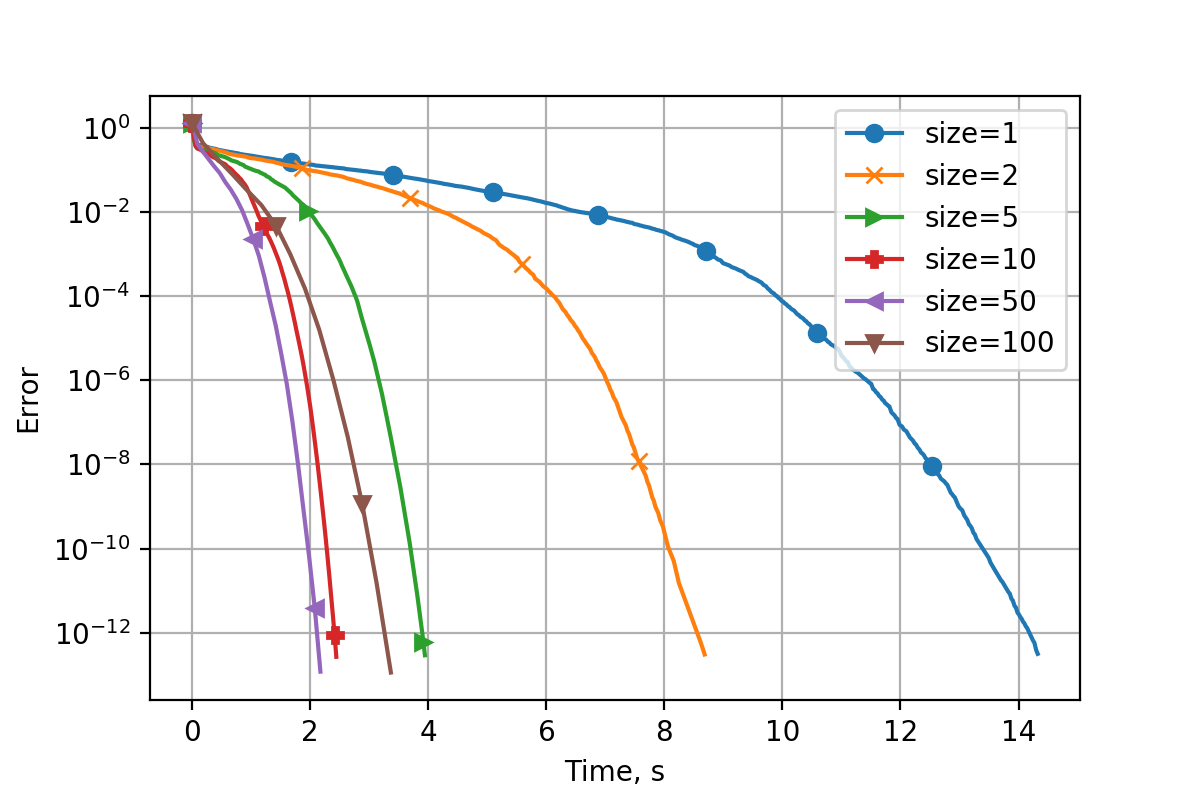}
		\caption{gauss}
	\end{subfigure}
	\begin{subfigure}[t]{0.23\textwidth}
		\centering
		\includegraphics[width = \textwidth]{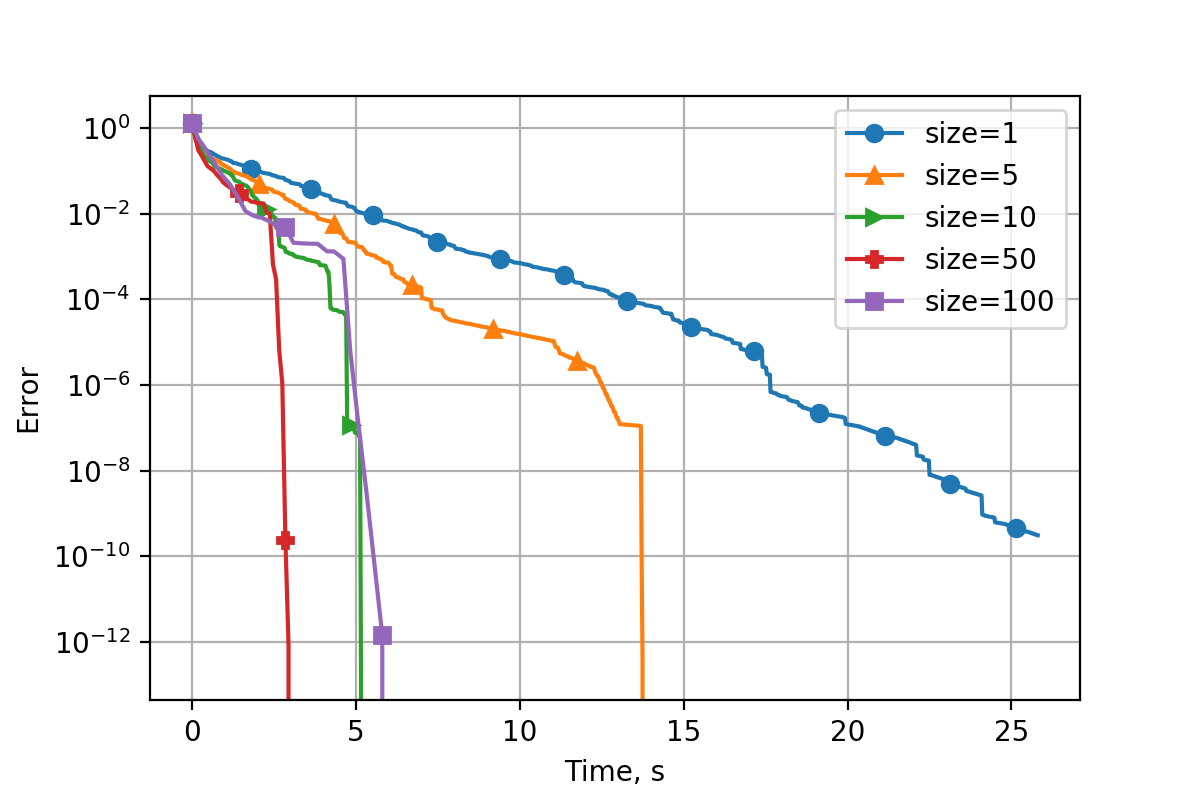}
		\caption{coord}
	\end{subfigure}
	\\
	\begin{subfigure}[t]{0.23\textwidth}
		\centering
		\includegraphics[width = \textwidth]{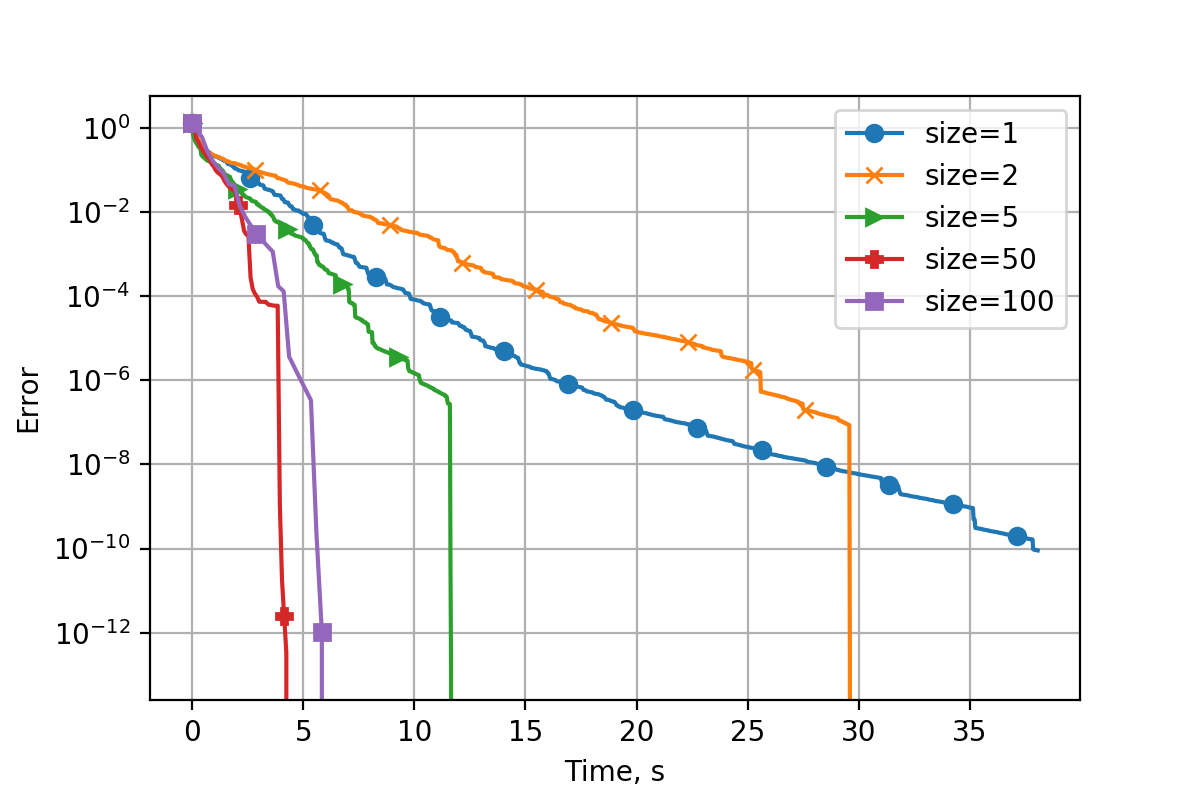}
		\caption{svd}
	\end{subfigure}
	\begin{subfigure}[t]{0.23\textwidth}
		\centering
		\includegraphics[width = \textwidth]{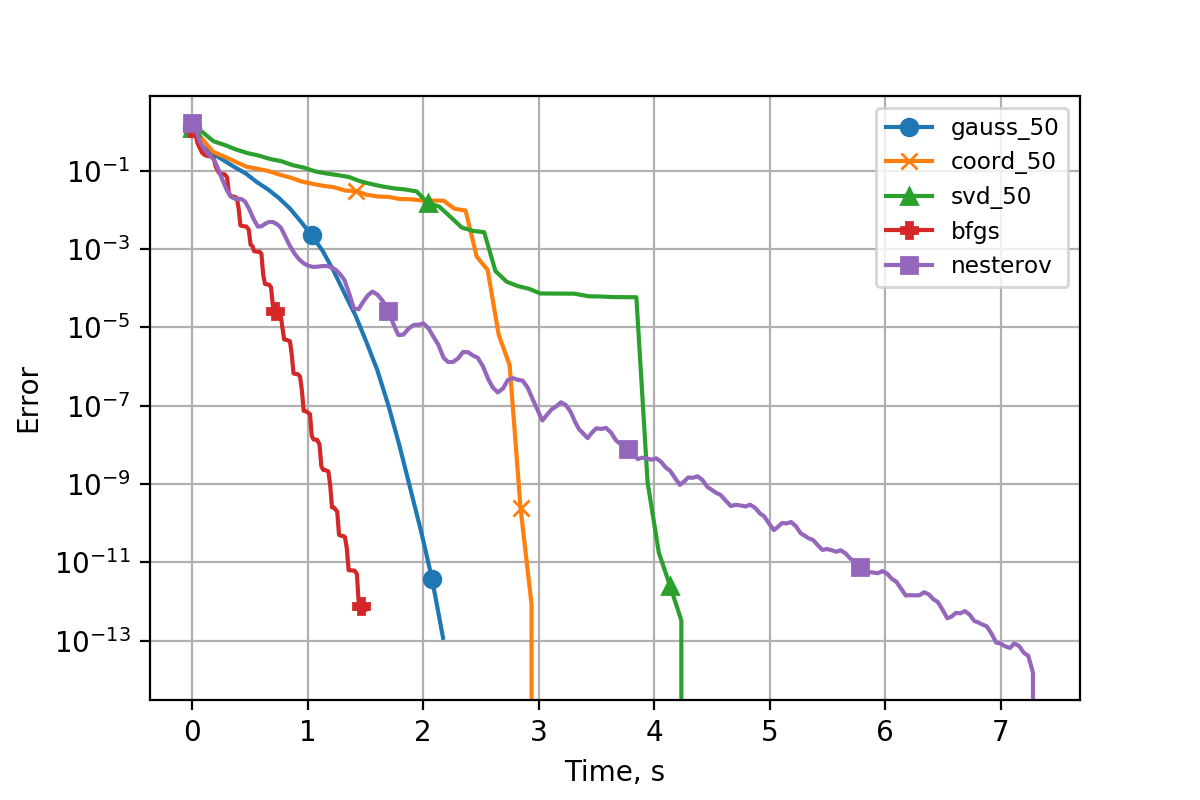}
		\caption{methods compared}
	\end{subfigure}
	\caption{w8a, $\lambda = 10^{-3}$. (n, d) = (49749, 300)}
	\label{fig:w8a}
\end{figure}

\begin{figure}
	\centering
	\begin{subfigure}[t]{0.23\textwidth}
		\centering
		\includegraphics[width = \textwidth]{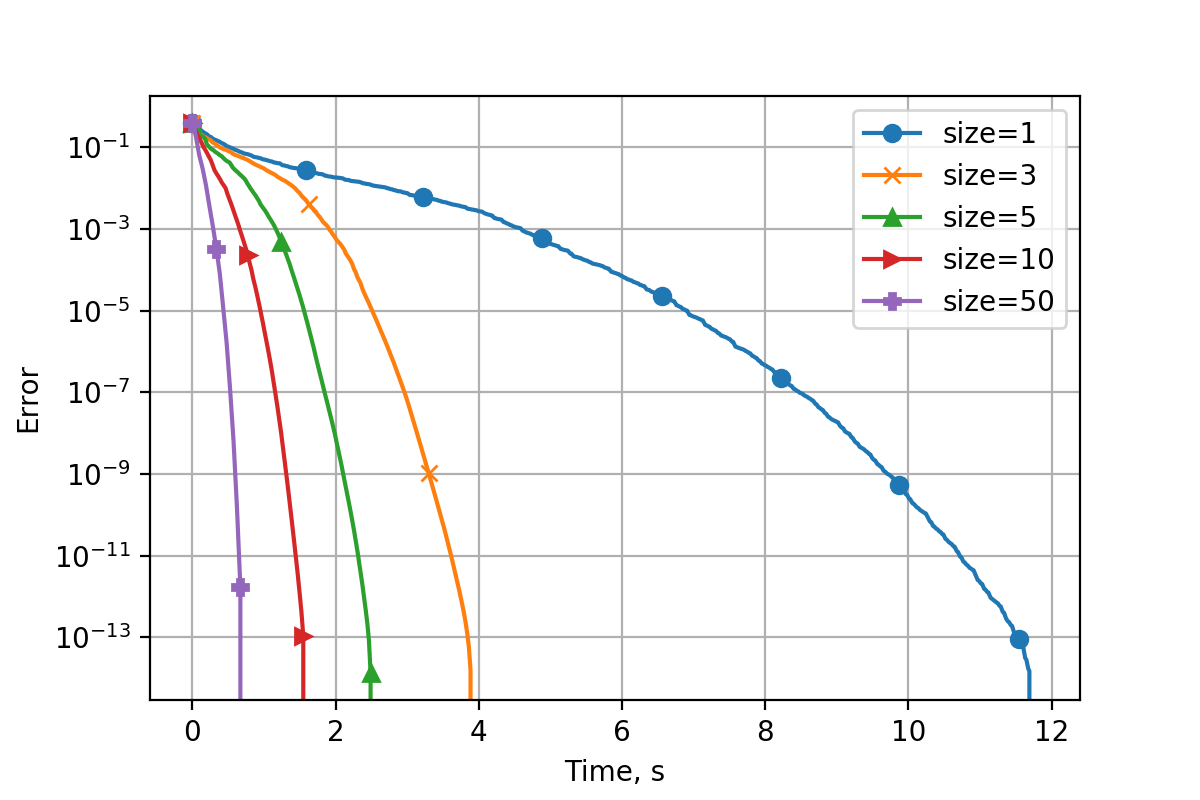}
		\caption{gauss}
	\end{subfigure}
	\begin{subfigure}[t]{0.23\textwidth}
		\centering
		\includegraphics[width = \textwidth]{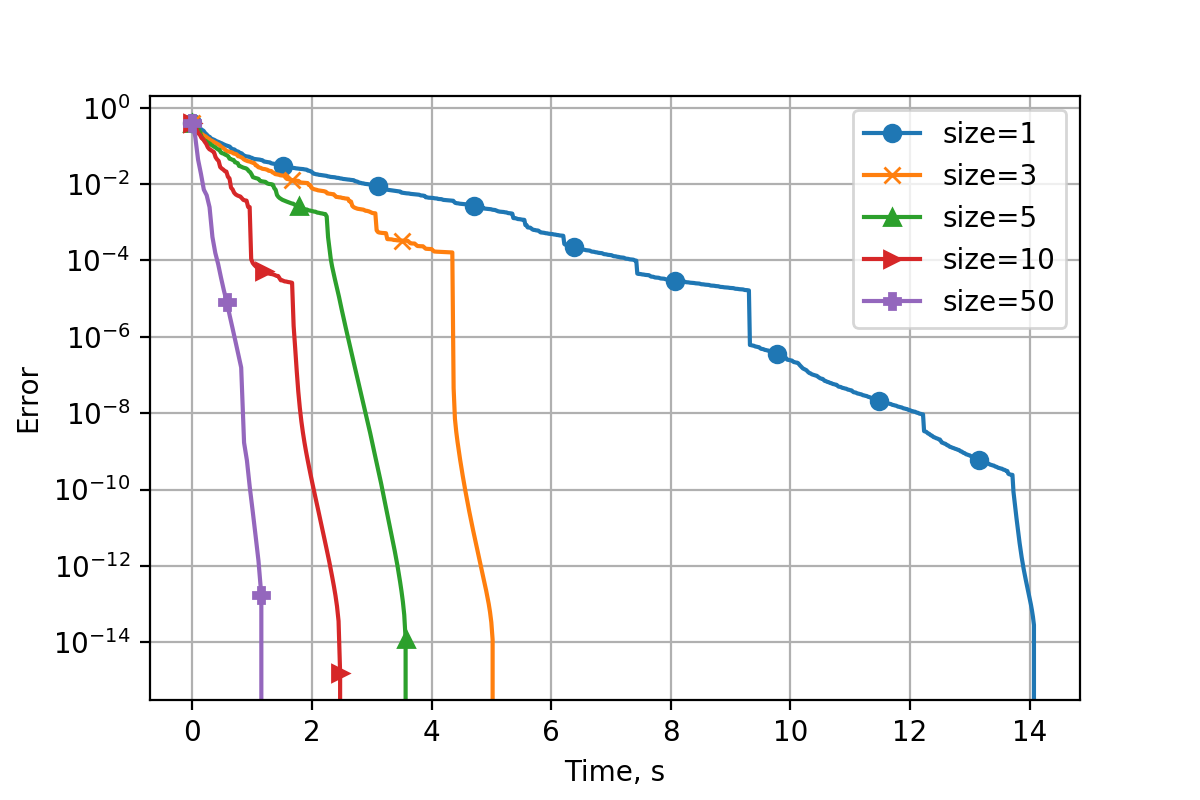}
		\caption{coord}
	\end{subfigure}
	\\
	\begin{subfigure}[t]{0.23\textwidth}
		\centering
		\includegraphics[width = \textwidth]{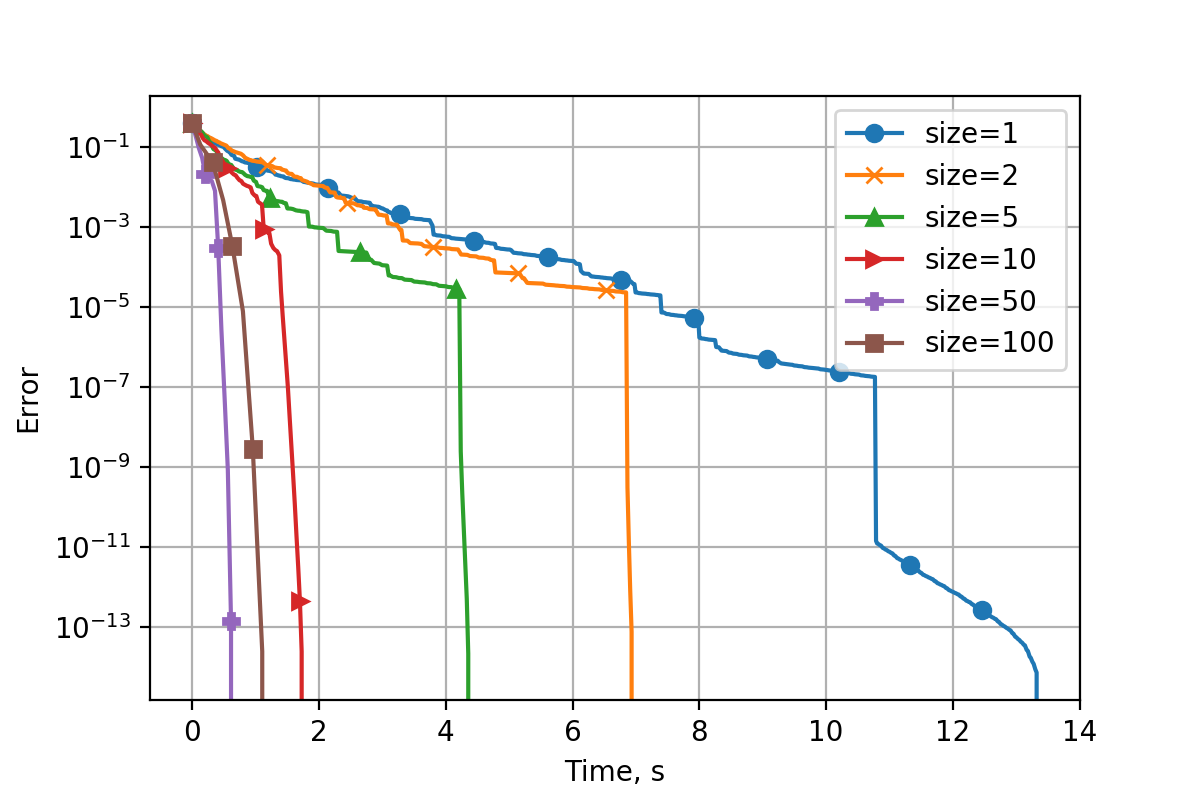}
		\caption{svd}
	\end{subfigure}
	\begin{subfigure}[t]{0.23\textwidth}
		\centering
		\includegraphics[width = \textwidth]{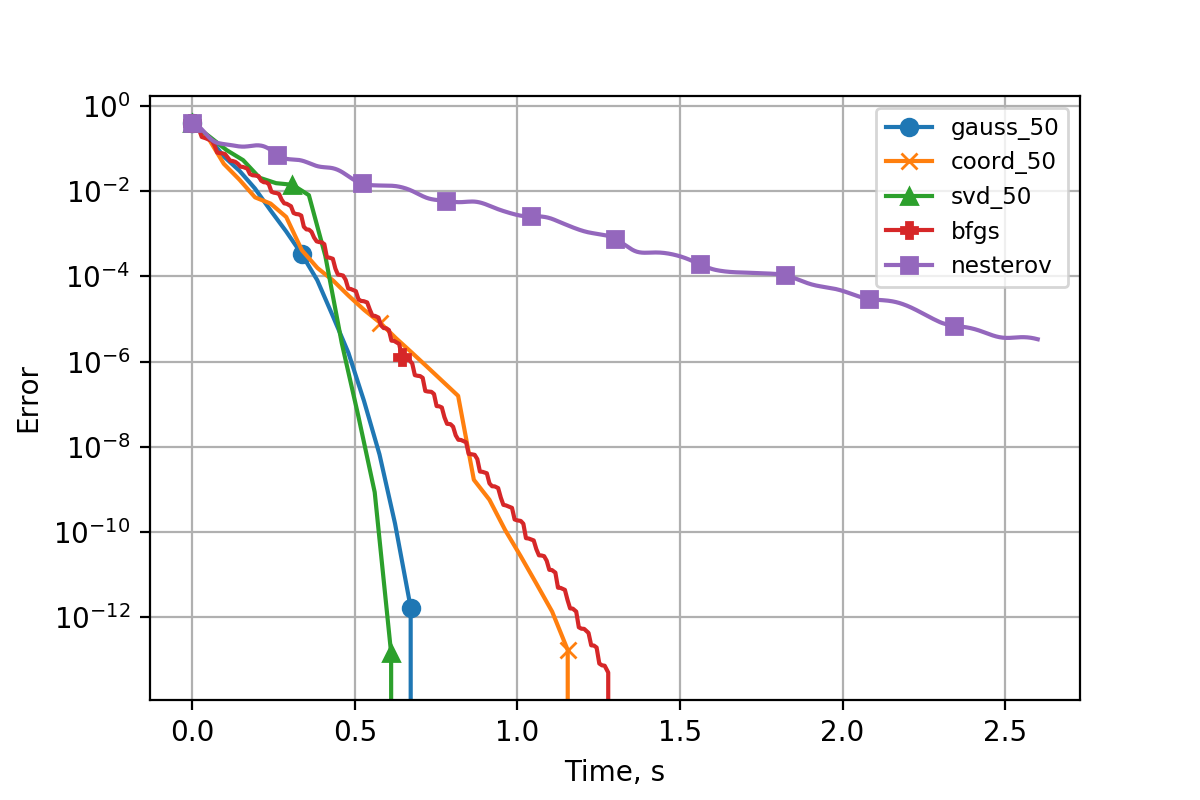}
		\caption{methods compared}
	\end{subfigure}
	\caption{a9a, $\lambda=10^{-3}$. (n, d) = (29159, 123)}
	\label{fig:a9a}
\end{figure}

\begin{figure}
	\centering
	\begin{subfigure}[t]{0.23\textwidth}
		\centering
		\includegraphics[width = \textwidth]{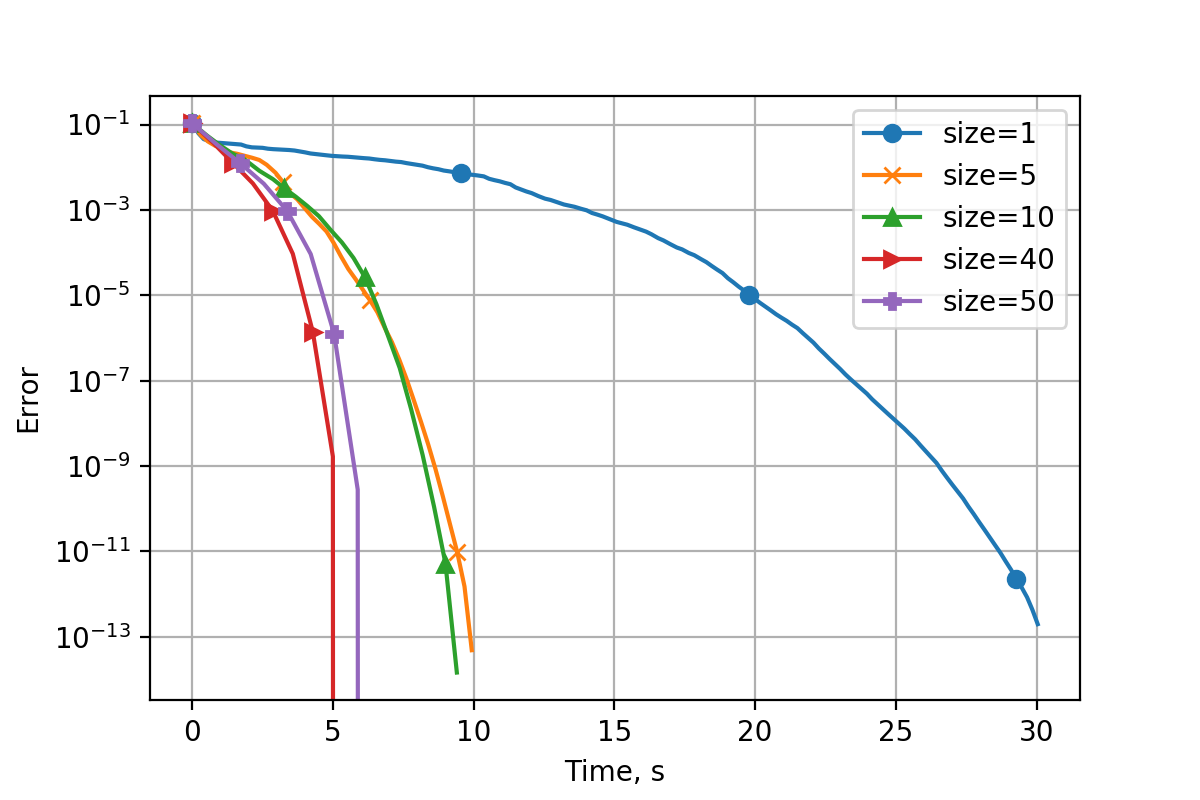}
		\caption{gauss}
	\end{subfigure}
	\begin{subfigure}[t]{0.23\textwidth}
		\centering
		\includegraphics[width = \textwidth]{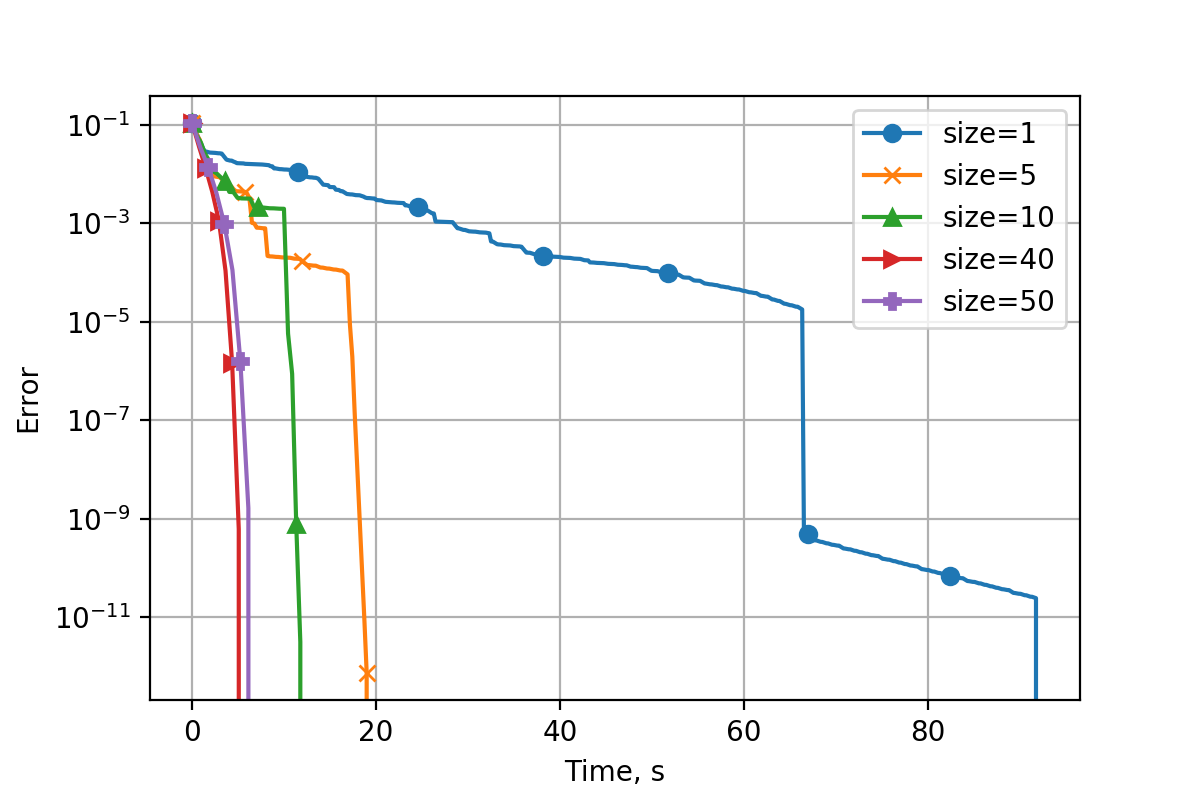}
		\caption{coord}
	\end{subfigure}
	\\
	\begin{subfigure}[t]{0.23\textwidth}
		\centering
		\includegraphics[width = \textwidth]{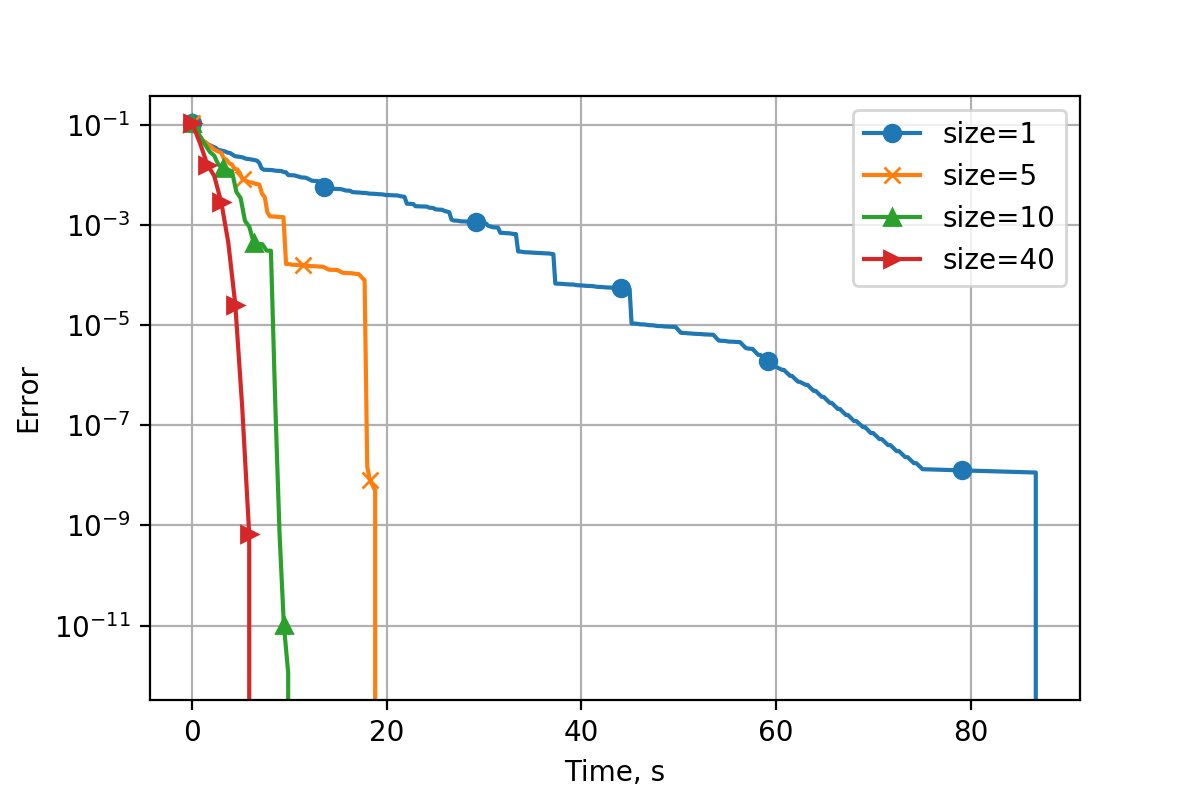}
		\caption{svd}
	\end{subfigure}
	\begin{subfigure}[t]{0.23\textwidth}
		\centering
		\includegraphics[width = \textwidth]{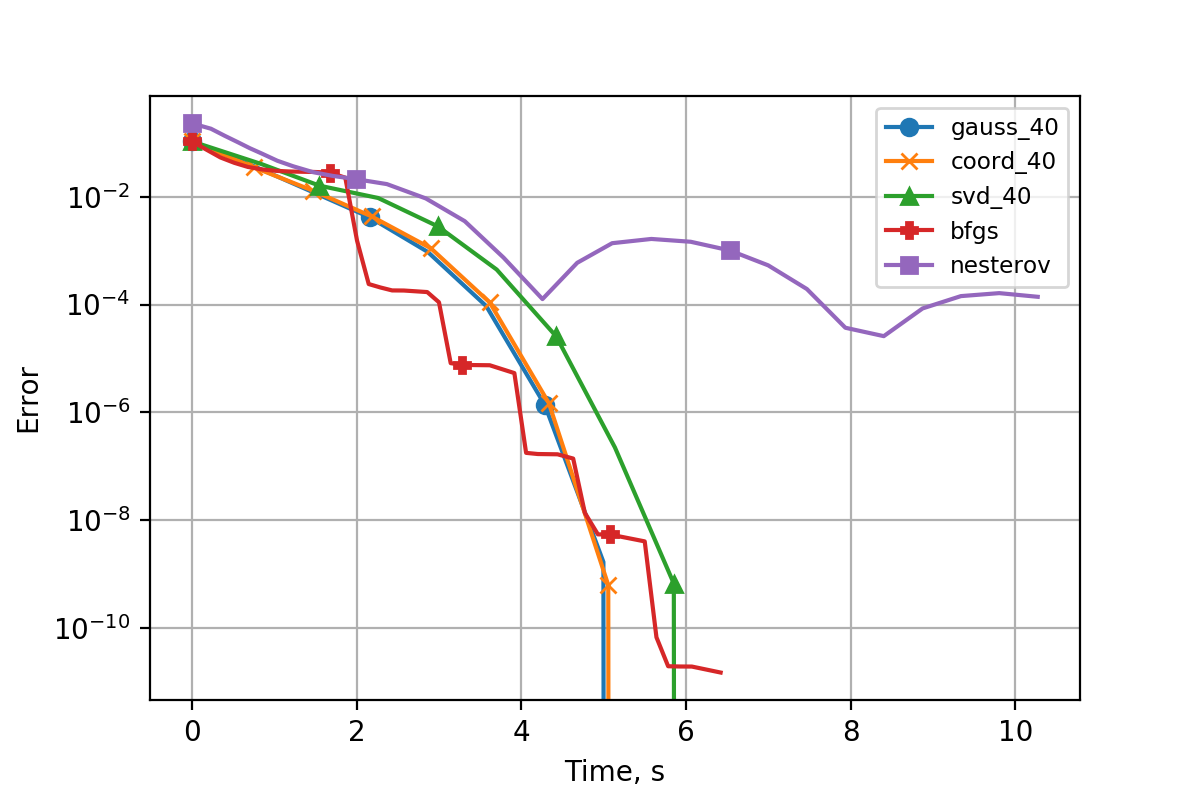}
		\caption{methods compared}
	\end{subfigure}
	\caption{covtype, $\lambda=10^{-3}$. (n, d) = (581012, 54)}
	\label{fig:covtype}
\end{figure}

\begin{figure}
	\centering
	\begin{subfigure}[t]{0.23\textwidth}
		\centering
		\includegraphics[width = \textwidth]{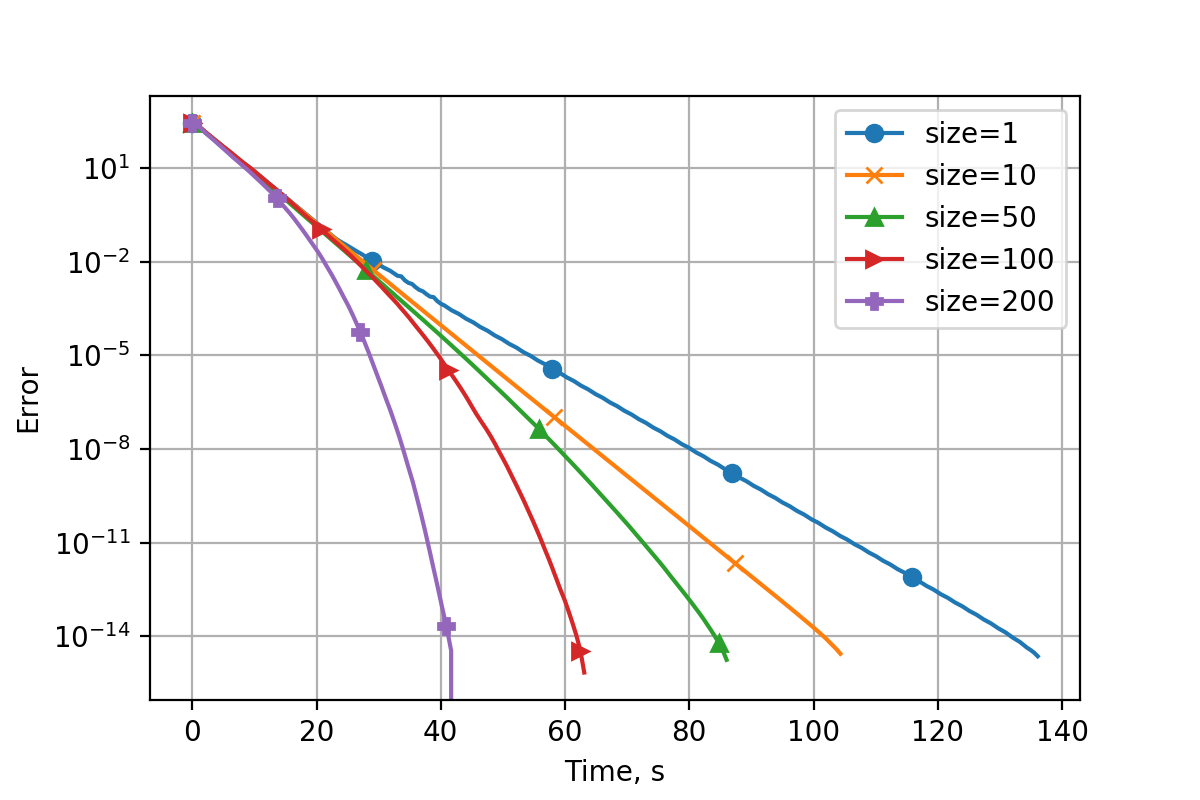}
		\caption{gauss}
	\end{subfigure}
	\begin{subfigure}[t]{0.23\textwidth}
		\centering
		\includegraphics[width = \textwidth]{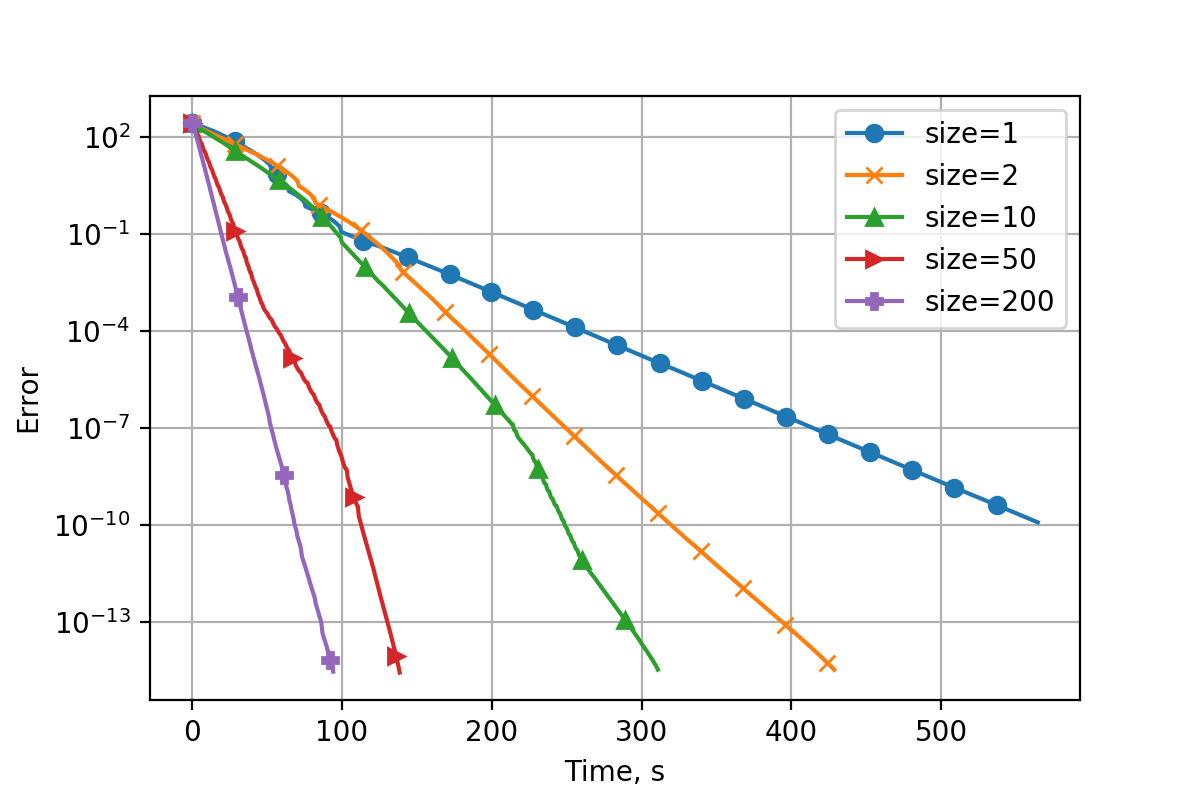}
		\caption{coord}
	\end{subfigure}
	\\
	\begin{subfigure}[t]{0.23\textwidth}
		\centering
		\includegraphics[width = \textwidth]{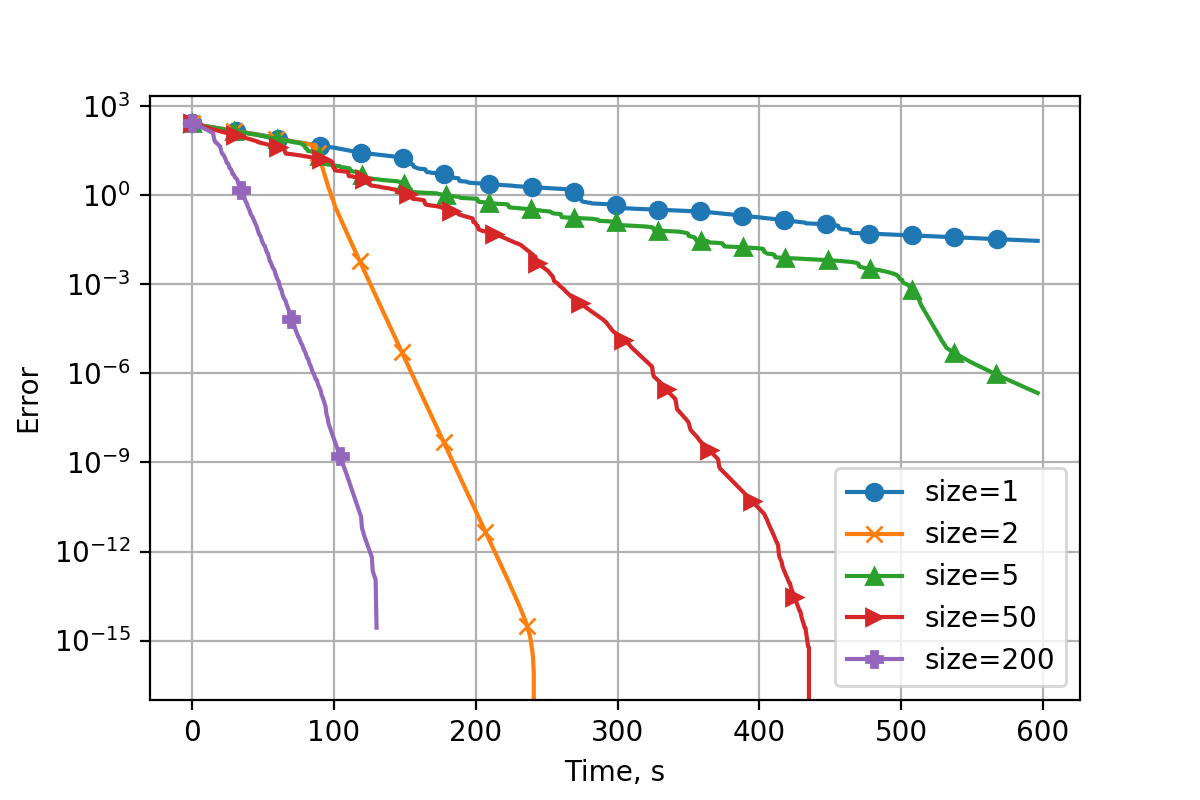}
		\caption{svd}
	\end{subfigure}
	\begin{subfigure}[t]{0.23\textwidth}
		\centering
		\includegraphics[width = \textwidth]{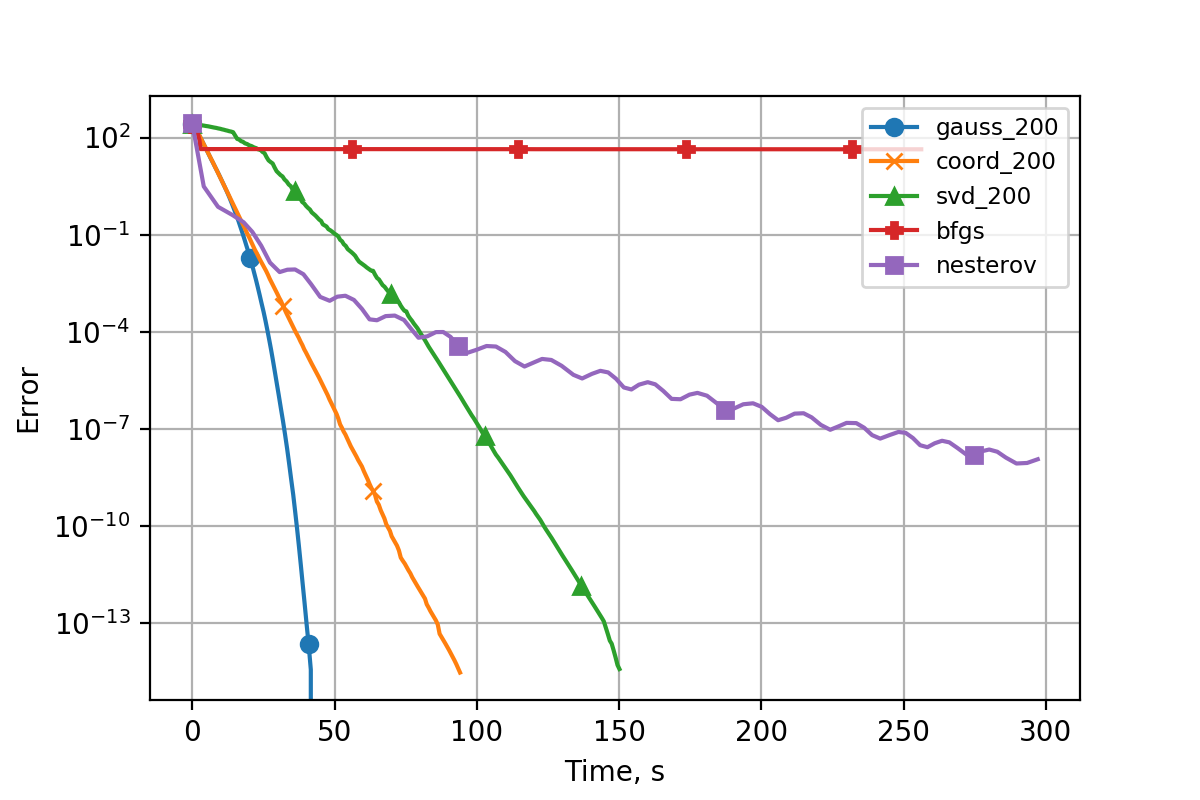}
		\caption{methods compared}
	\end{subfigure}
	\caption{gisette, $\lambda=10^{-1}$. (n, d) = (6000, 5000)}
	\label{fig:gisette}
\end{figure}

Contrary to the simulation results in Section~\ref{subsection:exper_synthetic}, we find that \textbf{svd} sketch is outperformed by others. Overall, \textbf{gauss} yielded the best results throughout all numerical tests. Randomized BFGS is comparable to the classical  BFGS method on datasets with small $d$ and moderate number of samples, see Figures~\ref{fig:w8a},~\ref{fig:a9a},~\ref{fig:covtype}, but it performs better than classical BFGS on datasets with a high number of samples; see Figures~\ref{fig:higgs},~\ref{fig:susy}. Moreover, Randomized BFGS significantly outperforms the classical BFGS method on larger dimensional problems such as {\tt gisette} ($d = 5,000$), {\tt colon-cancer} ($d = 2,000$) and {\tt epsilon} ($d = 2,000$); see Figures~\ref{fig:gisette},~\ref{fig:colon_cancer} and~\ref{fig:epsilon}.

\begin{figure}
	\centering
	\begin{subfigure}[h]{0.23\textwidth}
		\centering
		\includegraphics[width = \textwidth]{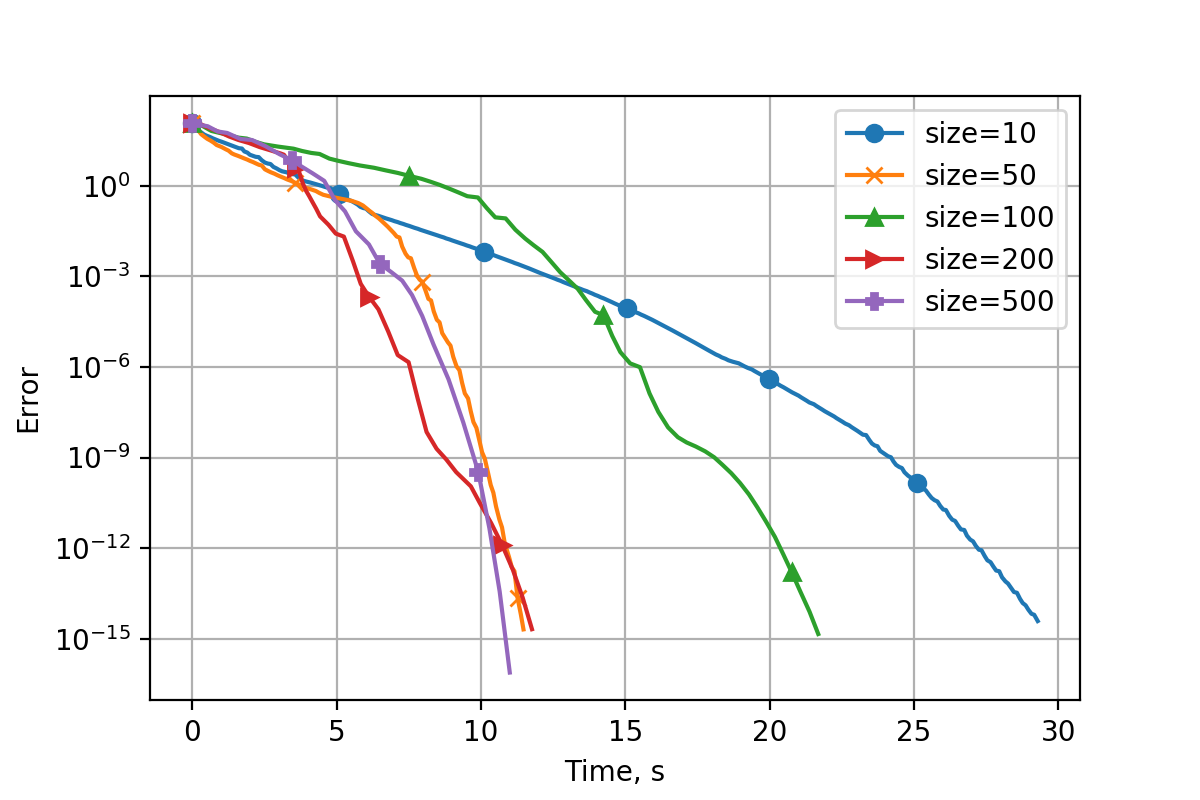}
		\caption{gauss}
	\end{subfigure}
	\begin{subfigure}[h]{0.23\textwidth}
		\centering
		\includegraphics[width = \textwidth]{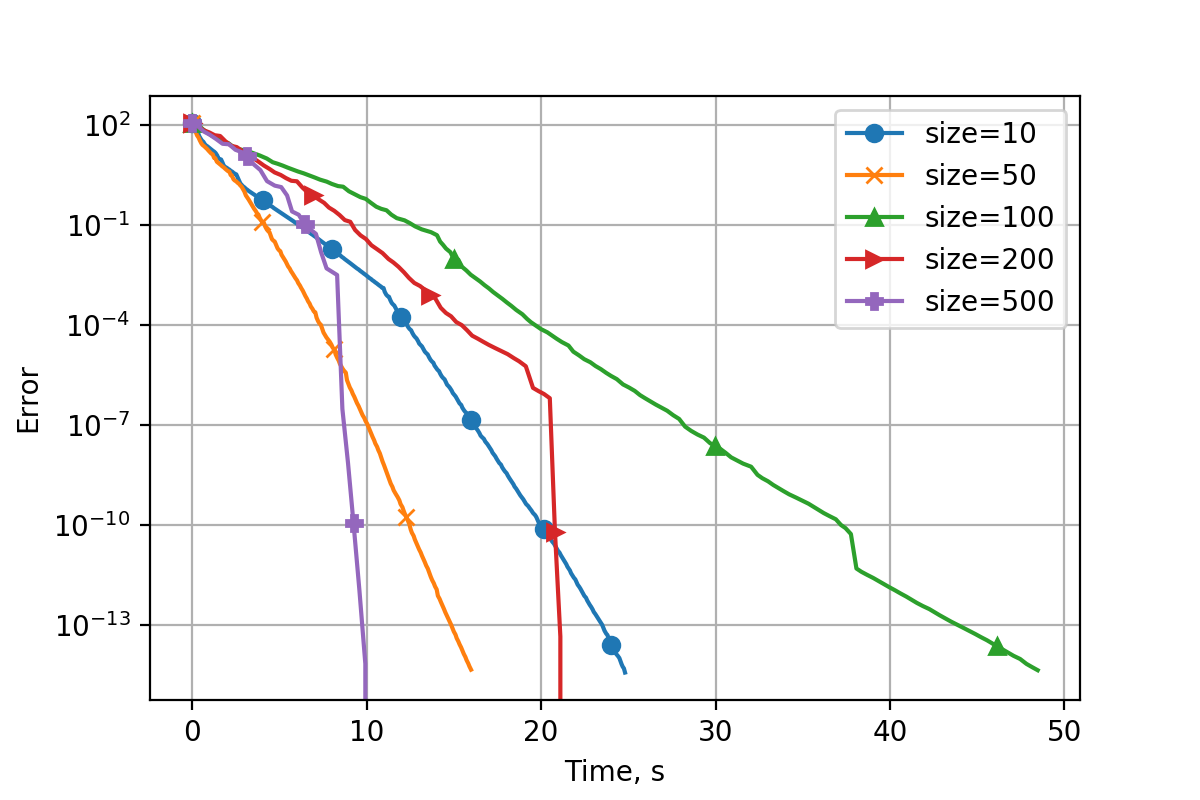}
		\caption{coord}
	\end{subfigure}
	\\
	\begin{subfigure}[h]{0.23\textwidth}
		\centering
		\includegraphics[width = \textwidth]{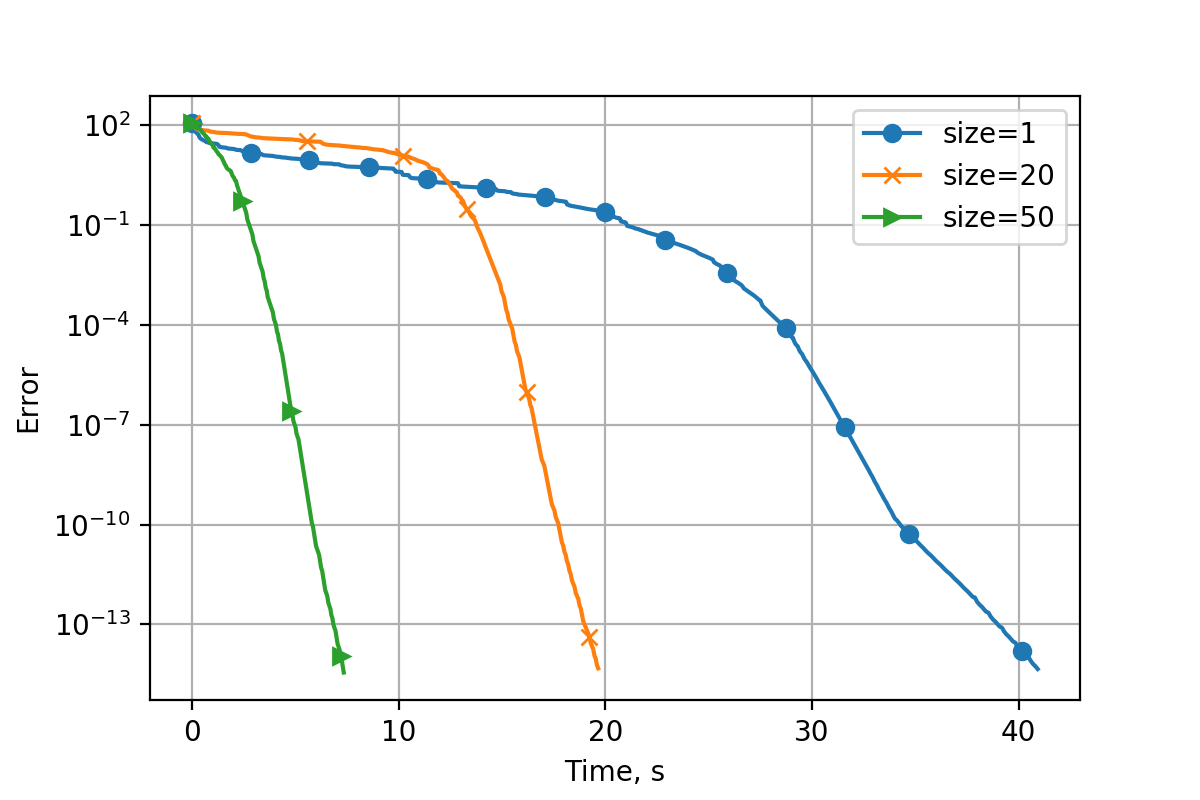}
		\caption{svd}
	\end{subfigure}
	\begin{subfigure}[h]{0.23\textwidth}
		\centering
		\includegraphics[width = \textwidth]{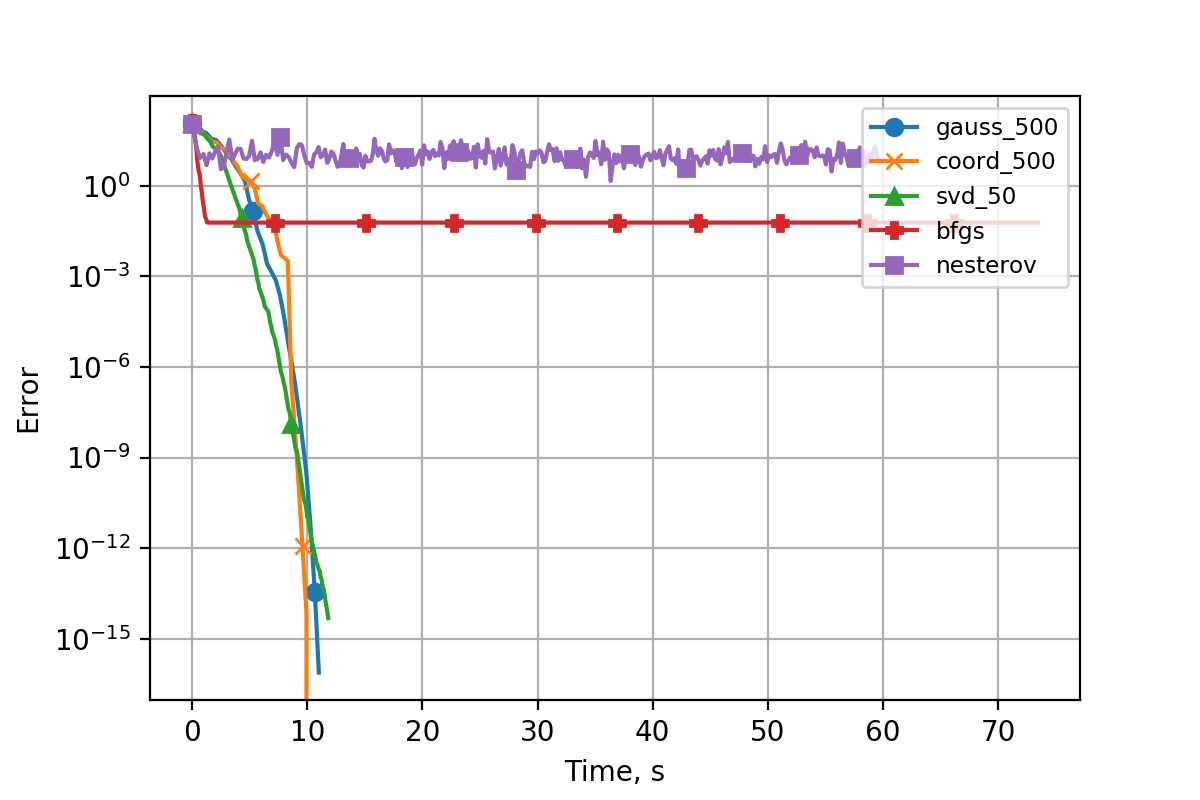}
		\caption{methods compared}
	\end{subfigure}
	\caption{colon-cancer, $\lambda=10^{-1}$. (n, d) = (62, 2000)}
	\label{fig:colon_cancer}
\end{figure}

\begin{figure}
	\centering
	\includegraphics[width = 0.4\textwidth]{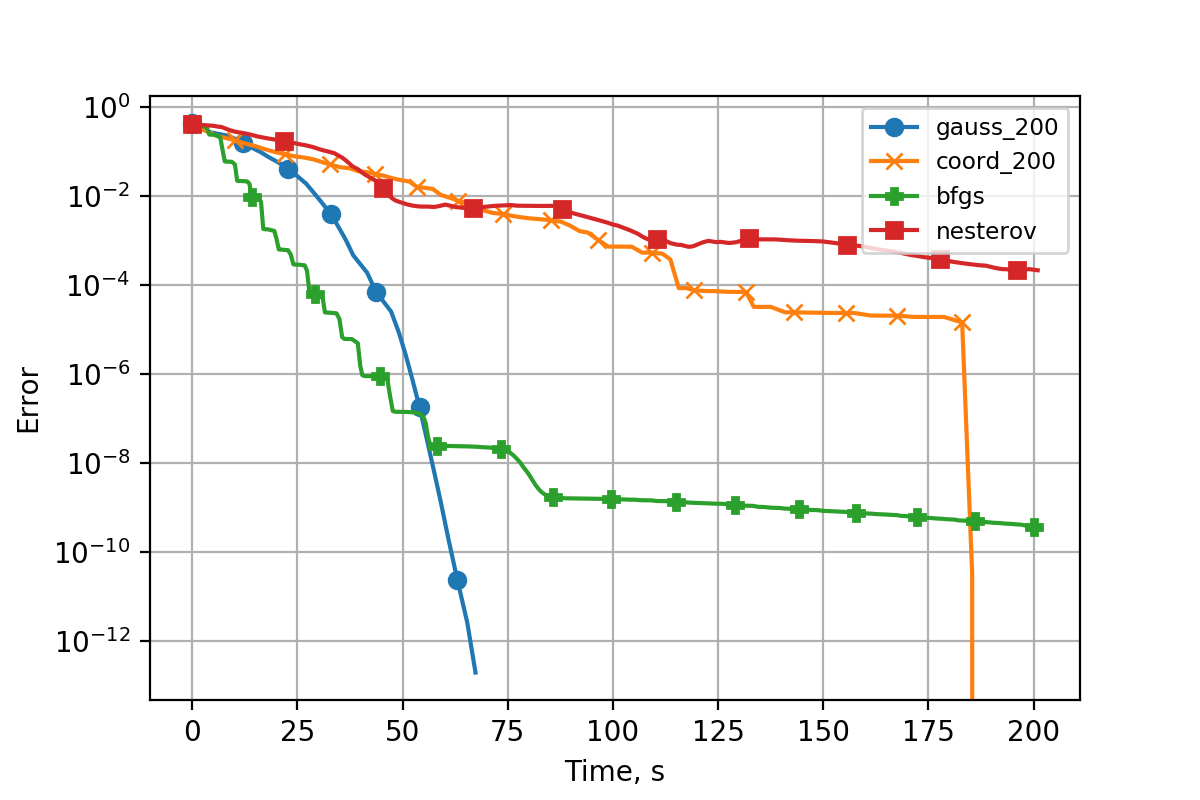}
	\caption{epsilon, $\lambda=10^{-4}$. (n, d) = (400000, 2000)}
	\label{fig:epsilon}
\end{figure}


\section{Conclusion, Consequences and Future Work}

With a meaningful rate for the local convergence of  randomized BFGS, we now point to several new open avenues for further study.

\paragraph{Recovering size of local convergence neighborhood of Newton's method.} In Section~\ref{sec:b98fg9809f} we show that when $\bS=\bI$, we recover Newton’s method and its rate of convergence. We also show that the size of the region of convergence is proportional to $1/d$. On the other hand, in the classic analysis of Newton’s method, the size of this region is $1/4$. This is a marked difference, and we believe it is because we not only establish the convergence of the iterates, but also the convergence of the inverse Hessians!
 It seems that showing that the Hessians also converge at a favorable linear rate has come at the cost of a smaller local region of convergence. We believe it makes sense to investigate whether this apparent deficiency is improvable or not.

\paragraph{Beyond access to sketched Hessian.} While RBFGS does not require access to the full Hessian, it requires access to a Hessian sketch. Note that the required Hessian sketch can correspond to as little as a single row of the Hessian only, and that this can be computed at the same cost as gradient evaluation (e.g., by doing backprop twice) as we pointed out in the paper. On the other extreme, when $\bS= \bI$, RBFGS  recovers the full Newton’s method. As such, RBFGS can be thought of as a randomized second-order method, or a method lying somewhere on the spectrum between a first-order method and second-order method. It should be interesting to investigate whether similar theory to ours can be established for some variants of RBFGS which do not require explicit sketch of the Hessian, but instead use difference of gradients as an approximation thereof. 

\paragraph{De-randomizing.} One immediate question is whether randomization is necessary to establish a good convergence rate for the BFGS method. That is, can we show that a (suitably constructed) deterministic variant of BFGS\footnote{Note that RBFGS becomes deterministic when $\bS =\bI$ with probability 1. However, this is not what we have in mind here. We are talking about deterministic variants RBFGS in cases when the sketching matrix $\mS$ has a small number of columns.} also enjoys a fast convergence rate? Based on our techniques, we believe it may be possible to achieve this through the following. First, we would need to establish a rate at which the estimates of the inverse matrix converge to a fixed constant Hessian matrix. In other words, we need to understand how fast would the BFGS matrices $\mB_k$ converge to the inverse Hessian were we to apply BFGS to minimizing a quadratic. We believe this step could perhaps be answered by  using the  equivalence of the BFGS method and the conjugate gradients method (CG)~\citep{Nazareth1979}. Since the CG method has been shown to converge linearly at an accelerated rate (as compared to gradient descent), we suspect that the Hessian estimates of the BFGS method might also converge at an accelerated rate. If this is confirmed, the remaining steps of the proof would follow verbatim from our proofs.

\paragraph{Globalizing.} By using a globalization strategy, such as the trust region framework~\citep{trustregionbook}, line-search or a carefully designed continuation scheme~\citep{Renegar:2001}, we believe it may be possible to extend our results to obtain a global convergence theory.

\paragraph{Practical consequences.} The general rate of convergence and the bound in~\eqref{eq:rhobndGLM} suggest that new sketching matrices could be devised  for accelerating our randomized BFGS method. We believe that by exploring the use of modern fast randomized sketches, such as the ROS sketch of~\citet{Pilanci2015a}, it may be possible to further accelerate the convergence of RBFGS.

\paragraph{Beyond self-concordance.}  It may be possible to extend  our results beyond the class of self-concordant functions to the class of generalized self-concordant functions~\citep{genscnewton,Bach2010} which also allow for one to control how fast the Hessian (and the inverse Hessian) changes.

%

\bibliography{stochQN-NeurIPS2020}
\bibliographystyle{icml2021}
\newpage

\appendix

\clearpage
\appendix

\onecolumn

\part*{Supplementary Material of Fast Linear Convergence of Randomized BFGS}

\section{Extra Experiments:  Figures~\ref{fig:susy} and~\ref{fig:higgs}}

\begin{figure}[h]
\centering
	\begin{subfigure}[t]{0.24\textwidth}
		\centering
		\includegraphics[width = \textwidth]{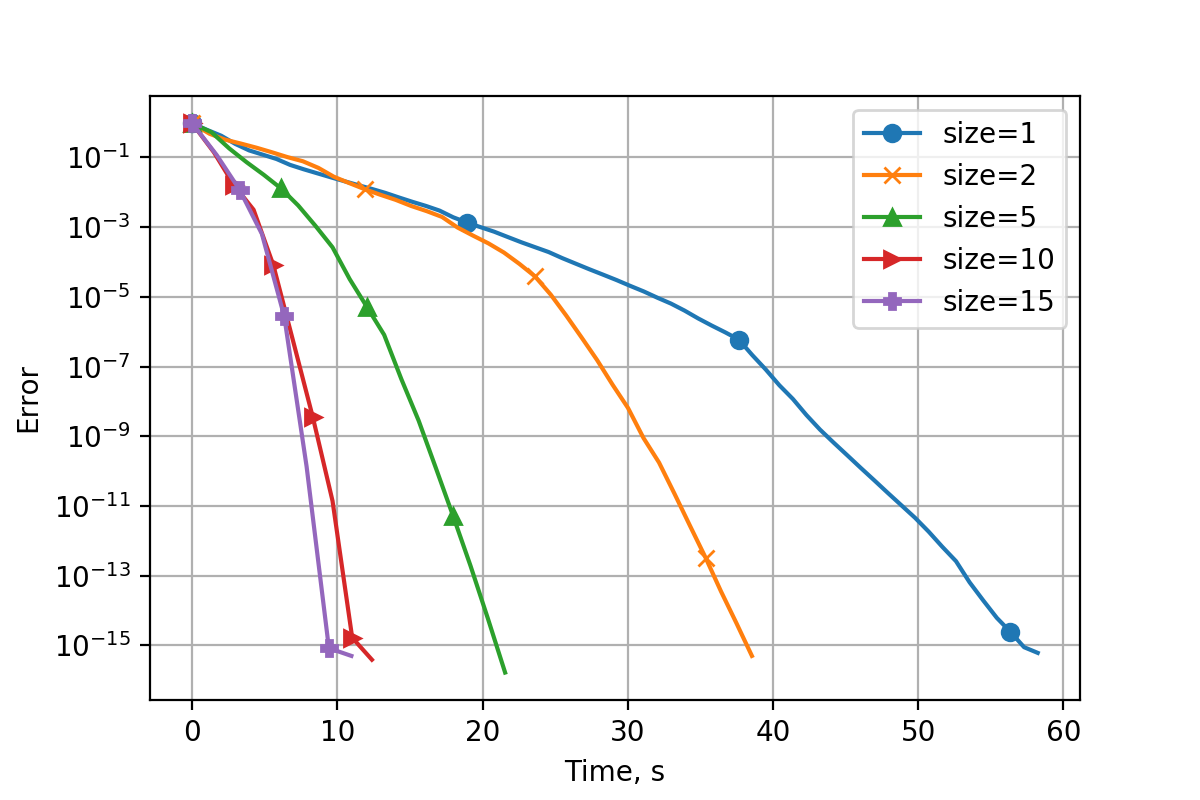}
		\caption{gauss}
	\end{subfigure}
	\begin{subfigure}[t]{0.24\textwidth}
		\centering
		\includegraphics[width = \textwidth]{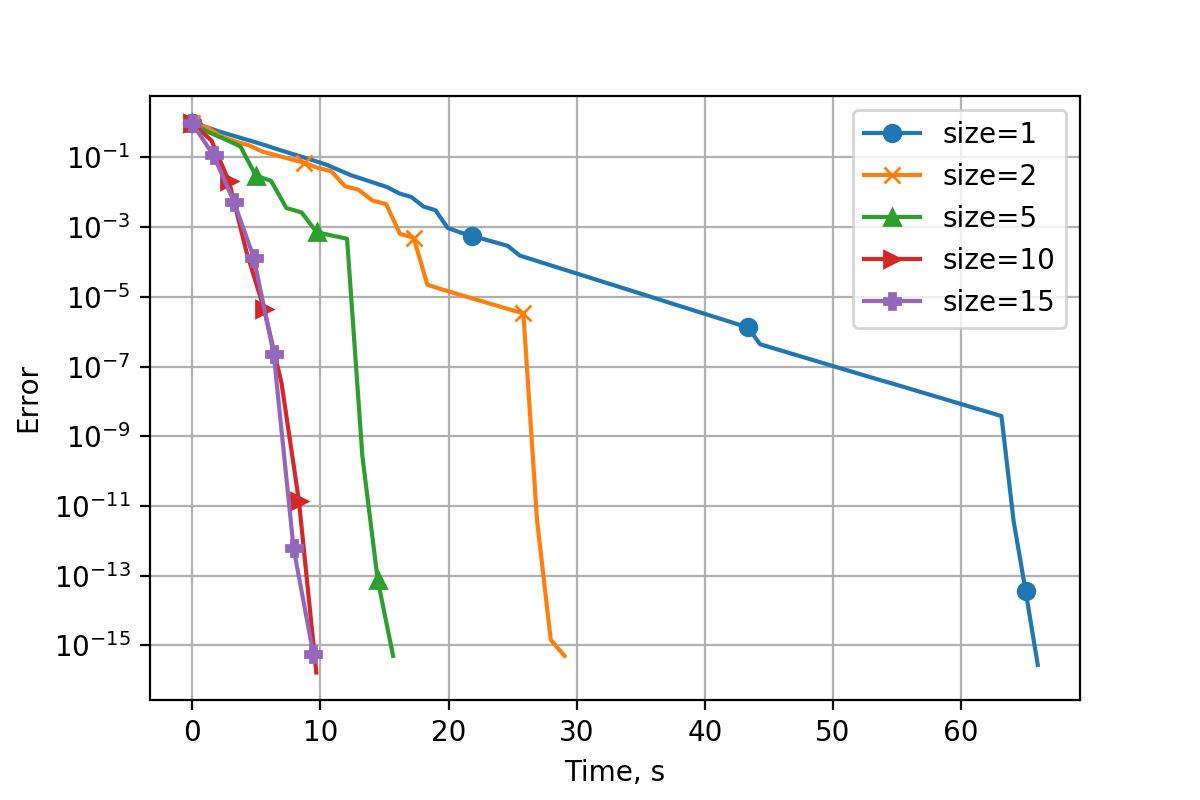}
		\caption{coord}
	\end{subfigure}
	\begin{subfigure}[t]{0.24\textwidth}
		\centering
		\includegraphics[width = \textwidth]{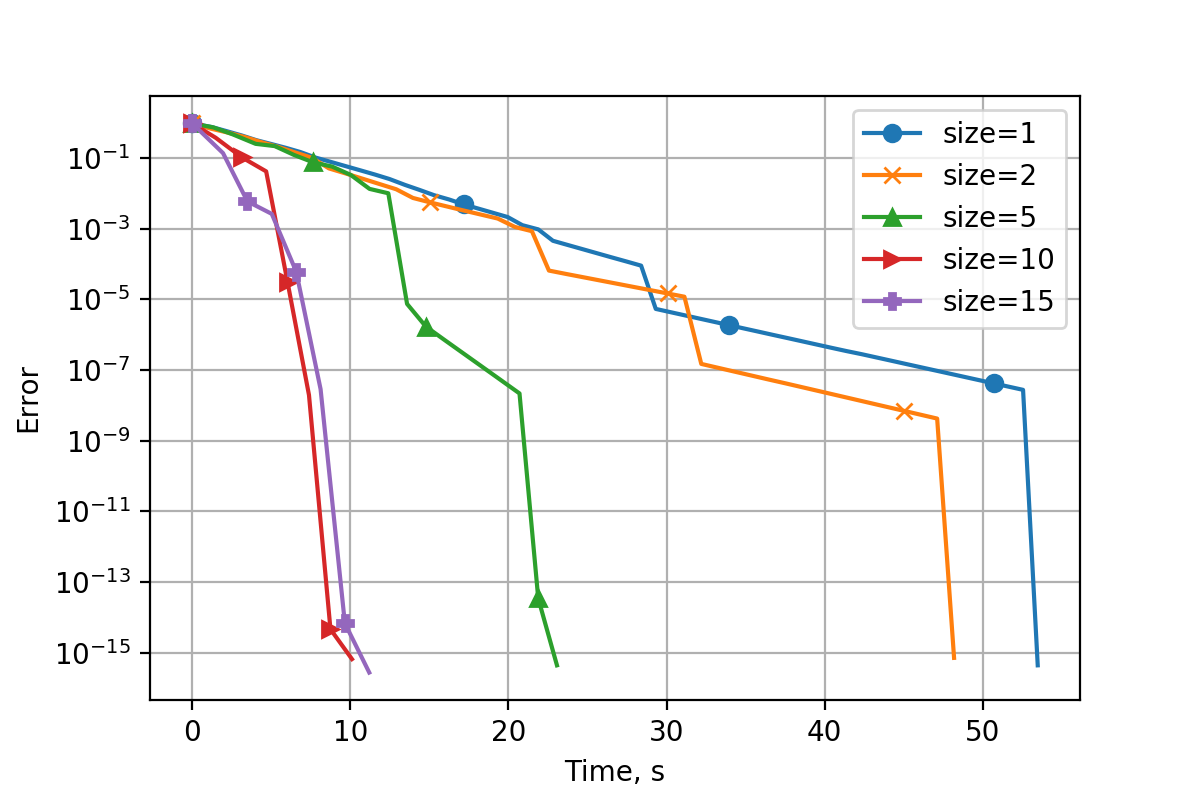}
		\caption{svd}
	\end{subfigure}
	\begin{subfigure}[t]{0.24\textwidth}
		\centering
		\includegraphics[width = \textwidth]{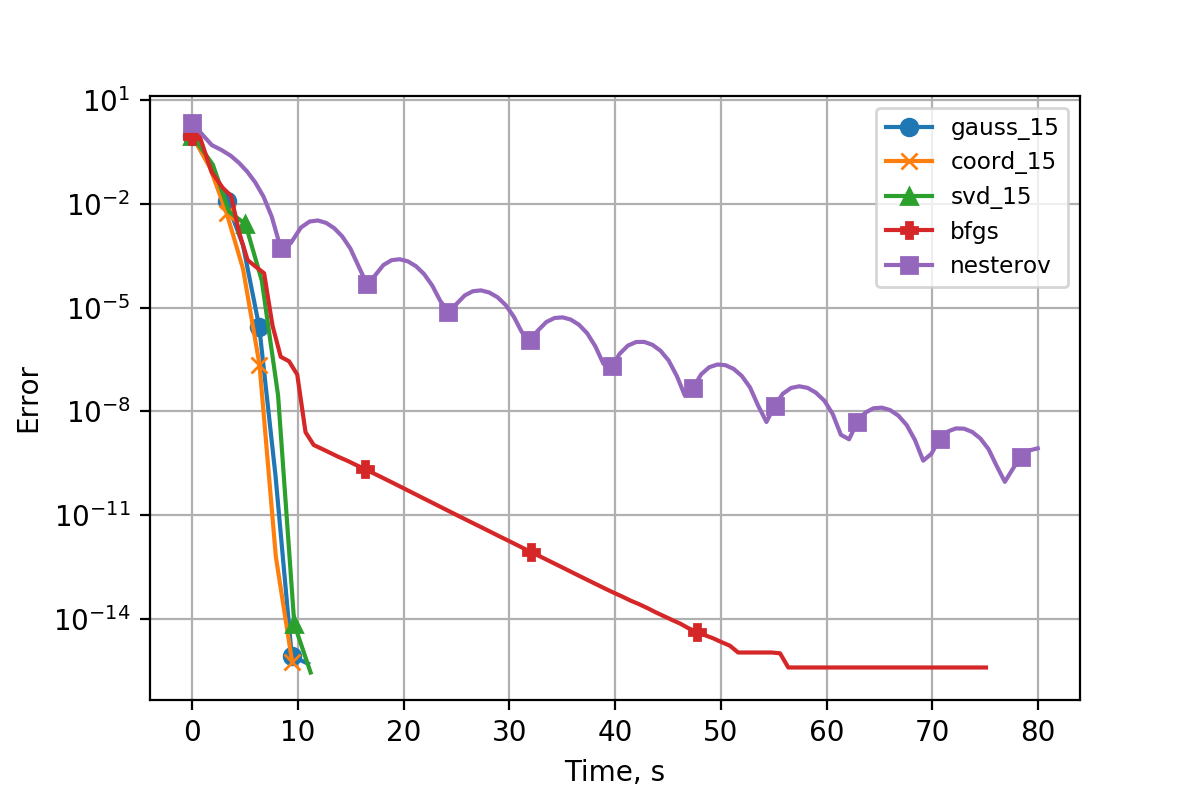}
		\caption{methods compared}
	\end{subfigure}
	\caption{{\tt SUSY}; $\lambda = 10^{-1}$; $n = 5,000,000$; $d = 18$; $\kappa = 6.1\cdot 10^3$}
	\label{fig:susy}
\end{figure}

\begin{figure}[!h]
		\begin{subfigure}[t]{0.24\textwidth}
			\centering
			\includegraphics[width = \textwidth]{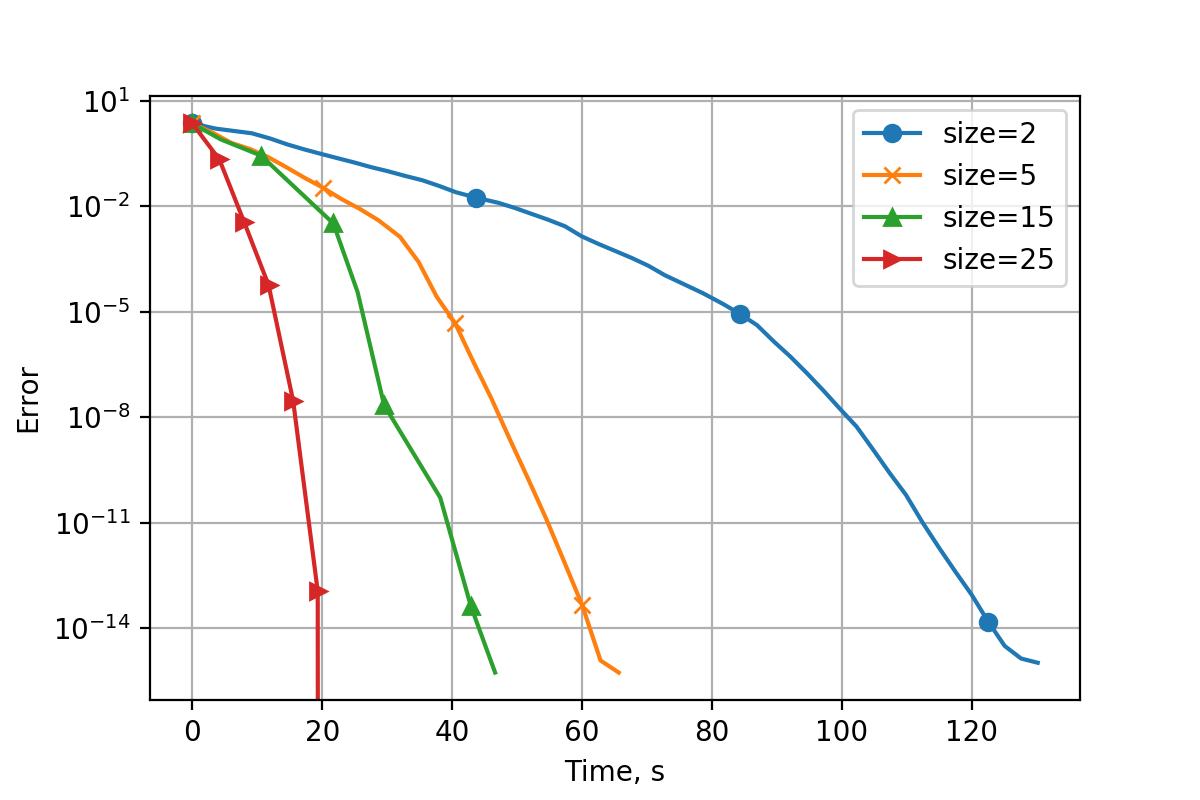}
			\caption{gauss}
		\end{subfigure}
		\begin{subfigure}[t]{0.24\textwidth}
			\centering
			\includegraphics[width = \textwidth]{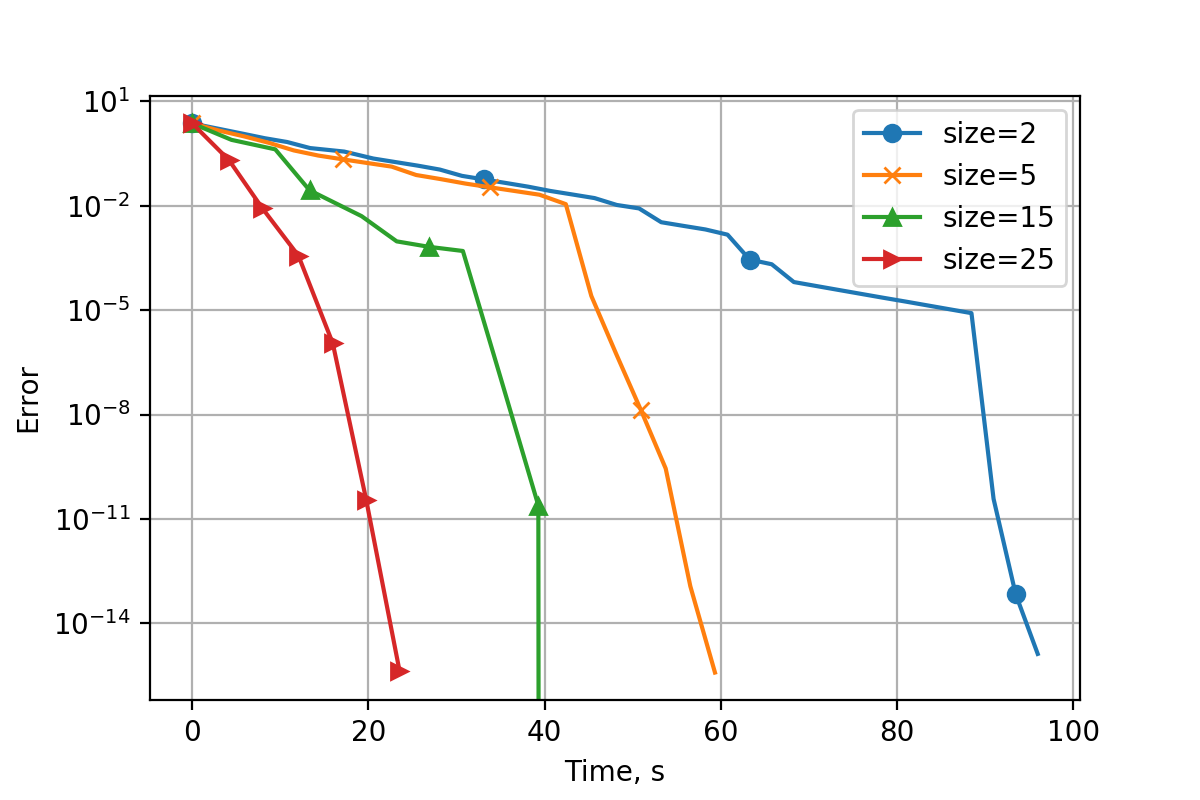}
			\caption{coord}
		\end{subfigure}
		\begin{subfigure}[t]{0.24\textwidth}
			\centering
			\includegraphics[width = \textwidth]{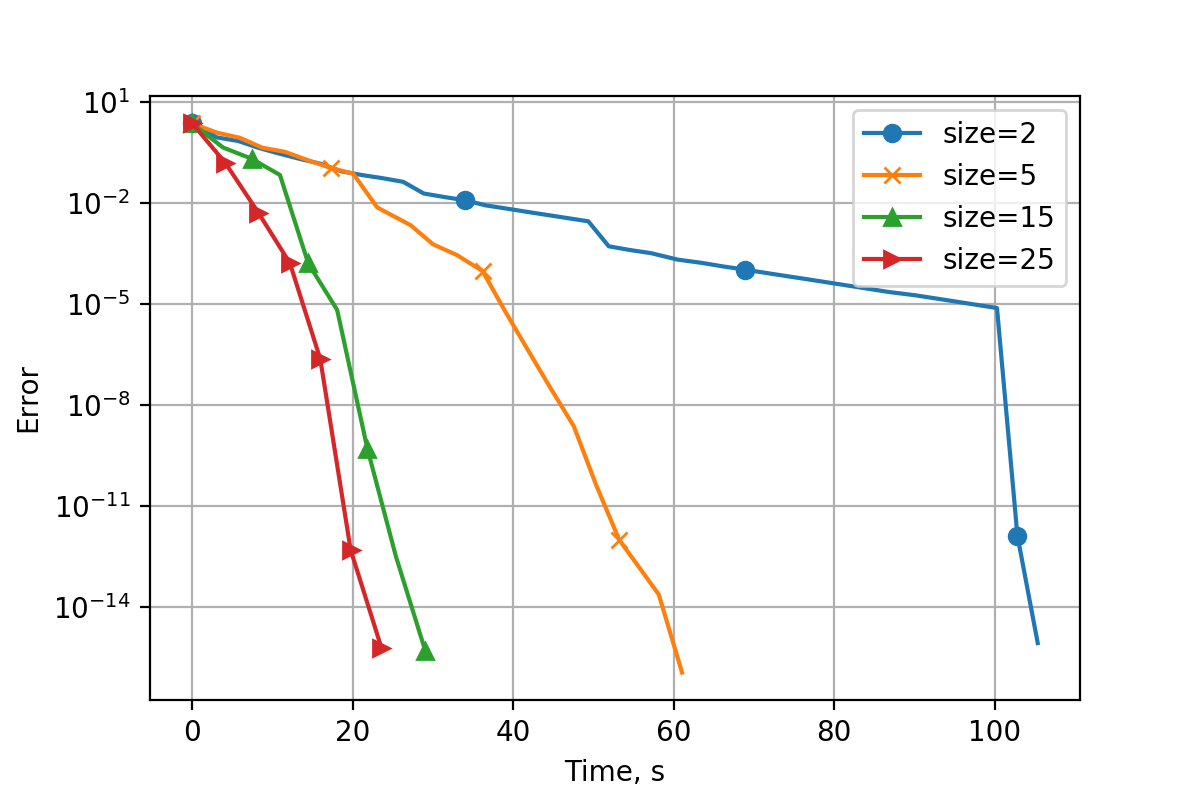}
			\caption{svd}
		\end{subfigure}
		\begin{subfigure}[t]{0.24\textwidth}
			\centering
			\includegraphics[width = \textwidth]{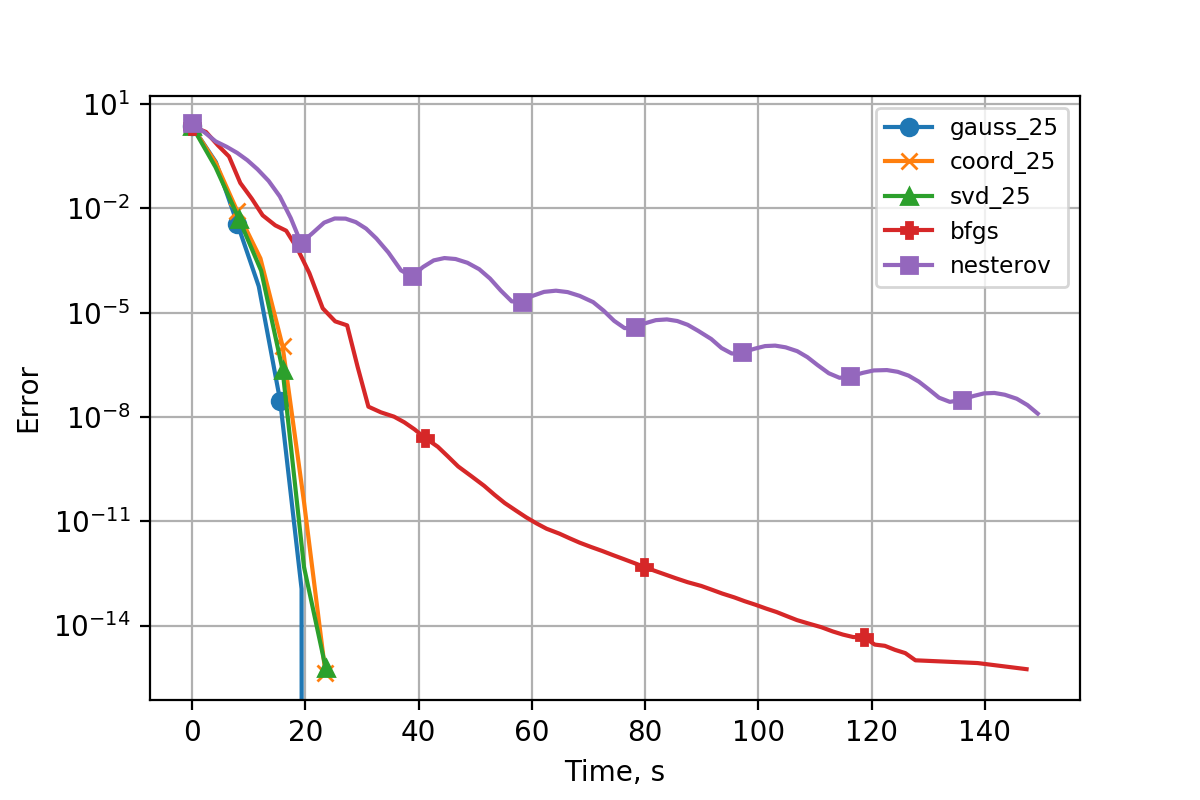}
			\caption{methods compared}
		\end{subfigure}
		\caption{HIGGS, $\lambda = 10^{-1}$. (n, d) = (11000000, 28)}
		\label{fig:higgs}
\end{figure}

\section{Proof of Theorem~\ref{theo:sc}}
\label{sec:prooftheosc}
First we collect some more notation.

\paragraph{Further notation.} For matrices $\mW \in \R^{d\times d}$, let $$ \norm{\mW}_x \eqdef \sup_{v \neq 0 } \frac{\norm{\mW v}_x}{\norm{v}_x}$$ denote the induced norm. Note that $\norm{\mW}_x = \norm{\mH_x^{1/2}\mW\mH_x^{-1/2}}_2$ and thus $\norm{\cdot}_x$ is sub-multiplicative. Indeed, 
\[ \norm{\mW \mV}_x = \norm{\mH_x^{1/2}\mW \mV\mH_x^{-1/2}}_2 =
\norm{\mH_x^{1/2}\mW \mH_x^{-1/2}\mH_x^{1/2}\mV\mH_x^{-1/2}}_2 \leq \norm{\mW }_x\norm{ \mV}_x.\] 
To further abbreviate our notation we will write  $\norm{\mW}_* \eqdef \norm{\mW}_{x_*} $. 

\subsection{Properties of self-concordant functions}

We now collect some known consequences of self-concordance and develop some additional properties that we need for our main proof.

\begin{lemma}\label{lem:hesscontrol} For all $y \in \cB_x^1$ we have that
\begin{equation} \label{eq:localhess}
 \norm{\mH_{x}^{-1} \mH_{y}}_x,  \;\norm{\mH_{y}^{-1} \mH_{x}}_x \quad \leq \quad
\frac{1}{\left(1 - \norm{y-x}_{x}\right)^2}, 
\end{equation}
and
\begin{equation} \label{eq:localhessdiff}
 \norm{\mI-\mH_{x}^{-1} \mH_{y}}_x, \;  \norm{\mI-\mH_{y}^{-1} \mH_{x}}_x \quad \leq \quad
\frac{1}{\left(1 - \norm{y-x}_{x}\right)^2}-1, 
\end{equation}
\end{lemma}
\begin{proof}
See Theorem 2.2.1 of \citet{Renegar:2001}.
\end{proof}

\begin{lemma}
If $\norm{x-x_*}_x \leq 1$ then 
  \begin{equation} \label{eq:NewtonstepHstarz}
 \norm{x - x_* - \mH_*^{-1}\nabla f(x)}_{*} \; \leq \; \frac{\norm{x-x_*}_*^2}{1-\norm{x-x_*}_*},
 \end{equation}
 and
 \begin{equation} \label{eq:Hinvegradbnd}
\norm{\mH_*^{-1}\nabla f(x_k)}_*  \; \leq \; \frac{\norm{x_*-x_k}_*}{1-\norm{x_k-x_*}_{*}}.
\end{equation}

\end{lemma}
\begin{proof}
Note that
\begin{eqnarray}
x - x_* - \mH_*^{-1}\nabla f(x) & = & x - x_* -\int_{0}^1\mH_*^{-1}\mH_{x_* +t(x-x_*)} (x-x_*)d t \nonumber \\
& = & \int_{0}^1\left(\mI -\mH_*^{-1}\mH_{x_* +t(x-x_*)}\right)(x-x_*)d t .\label{eq:nbiug7fudff}
\end{eqnarray}
Taking norms on both sides and using the sub-multiplicativity of the induced norm gives
\begin{eqnarray*}
\norm{x - x_* - \mH_*^{-1}\nabla f(x)}_{*} &\overset{\eqref{eq:nbiug7fudff}}{=}&\norm{\int_{0}^1\left(\mI -\mH_*^{-1}\mH_{x_* +t(x-x_*)}\right)(x-x_*)d t}_* \\
& \leq & \int_{0}^1\norm{\mI -\mH_*^{-1}\mH_{x_* +t(x-x_*)}}_* \norm{x-x_*}_*d t   \nonumber \\
& \overset{\eqref{eq:localhessdiff}}{\leq} &  \norm{x-x_*}_*\int_{0}^1\left(\frac{1}{\left(1 - t\norm{x-x_*}_{*}\right)^2}-1\right) d t  \nonumber \\
& = &  \frac{\norm{x-x_*}_*^2}{1 - \norm{x-x_*}_{*}}.
\end{eqnarray*}

The  proof of~\eqref{eq:Hinvegradbnd} follows from~\eqref{eq:localhess} since
\begin{equation*}
\norm{\mH_*^{-1}\nabla f(x)}_*    \leq \norm{x_*-x}_*  \int_{0}^1\norm{\mH_*^{-1}\mH_{x_* +t(x-x_*)}}_*  dt 
 \overset{\eqref{eq:localhess}}{\leq}  \int_{0}^1 \frac{\norm{x_*-x}_*}{(1 - t\norm{x-x_*}_{*})^2} dt  
 = \frac{\norm{x_*-x}_*}{1-\norm{x-x_*}_{*}}.
\end{equation*}
\end{proof}

Using self-concordance we can change the metric in the weighted Frobenius.
\begin{lemma}
	Let $\mW \in \sym^d$.
	For all $y \in \cB_x^1$ the following inequality holds:
	\begin{equation}\label{eq:normchangescx}
		\normfro{\mW}{\mH_{y}} \leq  \frac{\normfro{\mW}{\mH_{x}}}{\left(1 - \norm{y-x}_{x}\right)^2} .
	\end{equation}
	Furthermore, if $x \in \cB_y^{1/2}$ then 
		\begin{equation}\label{eq:normchangescy}
		\normfro{\mW}{\mH_{y}} \leq  \left(\frac{1- \norm{y-x}_{y}}{1-2 \norm{y-x}_{y}}\right)^2 \normfro{\mW}{\mH_{x}} .
	\end{equation}
\end{lemma}

\begin{proof}
	\begin{eqnarray*}
		\normfro{\mW}{\mH_{y}}
		&\leq &
		\normfro{\mW}{\mH_{x}}\norm{\mH_{x}^{-1/2} \mH_{y} \mH_{x}^{-1/2}}_2\\
		&= &	\normfro{\mW}{\mH_{x}}\norm{ \mH_{y} \mH_{x}^{-1}}_x \\
		& \overset{\eqref{eq:localhess}}{\leq} &   \frac{\normfro{\mW}{\mH_{x}}}{\left(1 - \norm{y-x}_{x}\right)^2} .
	\end{eqnarray*}
	The bound~\eqref{eq:normchangescy} follows from the fact that for $x \in \cB_y^{1/2}$  we have that $y \in \cB_x^1$ and the definition of self-concordance~\eqref{eq:selfcon} that
	\[ \norm{y-x}_{x} \leq \frac{ \norm{y-x}_{y}}{1- \norm{y-x}_{y}} \qquad \Rightarrow \qquad \frac{1}{1-\norm{y-x}_x} \leq \frac{1- \norm{y-x}_{y}}{1-2 \norm{y-x}_{y}}.\]

\end{proof}

\subsection{The distance of the iterates}

Next we need an upper bound on $\norm{x_{k+1}-x_*}_{*}^2$.

\begin{lemma} \label{lem:iteratescbnd}  If $x \in \cB_{x_*}^1$ then

\begin{align}
	\norm{x_{k+1} - x_*}_* & \leq \frac{\norm{x_k-x_*}_*^2}{2}\frac{3-2\norm{x_k-x_*}_*}{(1-\norm{x_k-x_*}_*)^2}+\frac{1}{2}\normfro{\mB_k - \mH_*^{-1}}{\mH_*}^2.
\end{align}
Consequently if $x \in \cB_{x_*}^{1/4}$ we have that
\begin{align}
	\norm{x_{k+1} - x_*}_* & \leq 5\norm{x_k-x_*}_*^2+ \frac{1}{2}\normfro{\mB_k - \mH_*^{-1}}{\mH_*}.\label{eq:iteratescbnd14}
\end{align}

\end{lemma}

\begin{proof}
It starts with
\begin{eqnarray*}
	\norm{x_{k+1} - x_*}_*
	&=&
	\norm{x_k - x_* - \mB_k\nabla f(x_k)}_* \nonumber \\
	&=&
	\norm{x_k - x_* - \mH_*^{-1}\nabla f(x_k)  + (\mH_*^{-1} - \mB_k)\nabla f(x_k)}_{*} \nonumber  \\
	&\leq &
	\norm{x_k - x_* - \mH_*^{-1}\nabla f(x_k)}_*
	+
	\norm{(\mH_*^{-1} - \mB_k)\nabla f(x_k)}_*\nonumber \\
	&\leq &
	\frac{\norm{x_k-x_*}_*^2}{1-\norm{x_k-x_*}_*}
	+
	\norm{(\mB_k-\mH_*^{-1} )\nabla f(x_k)}_* , \label{eq:s9k4ks9ks4}
	\end{eqnarray*}
where we used~\eqref{eq:NewtonstepHstarz} in the last step.
As for the second term we have that
\begin{eqnarray*}
\norm{(\mB_k- \mH_*^{-1})\nabla f(x_k)}_*  & = &
\norm{(\mB_k\mH_*- \mI)\mH_*^{-1}\nabla f(x_k)}_*  \nonumber \\
& \leq &  \norm{\mB_k\mH_*- \mI}_* \norm{\mH_*^{-1}\nabla f(x_k)}_* \nonumber \\
& \overset{\eqref{eq:Hinvegradbnd}}{\leq} &  \normfro{\mB_k - \mH_*^{-1}}{\mH_*}\frac{\norm{x_k-x_*}_{*}}{1- \norm{x_k-x_*}_{*}},
\end{eqnarray*}
where we used that 
\[  \norm{\mB_k\mH_*- \mI}_*  =  \norm{\mH_*^{1/2}\mB_k\mH_*^{1/2}- \mI}_2 \leq  \norm{\mH_*^{1/2}\mB_k\mH_*^{1/2}- \mI}_F =  \normfro{\mB_k- \mH_*^{-1}}{\mH_*}.\]
Finally using that $ab \leq \frac{a^2}{2} + \frac{b^2}{2}$ for all $a,b>0$ with $a = \normfro{\mB_k - \mH_*^{-1}}{\mH_*}$ and $b =\frac{\norm{x_k-x_*}_{*}}{1- \norm{x_k-x_*}_{*}}$ gives
\begin{align}
\norm{(\mB_k- \mH_*^{-1})\nabla f(x_k)}_*  
& \leq  \frac{1}{2}\normfro{\mB_k - \mH_*^{-1}}{\mH_*}^2 + \frac{1}{2}\frac{\norm{x_k-x_*}_{*}^2}{(1- \norm{x_k-x_*}_{*})^2}. \label{eq:sae89js84js}
\end{align}
The above combined with~\eqref{eq:s9k4ks9ks4} gives the result.
\end{proof}

\subsection{The distance of the quasi-Newton matrix}

We  start by establishing a lemma.
\begin{lemma} \label{lem:Bkplus1contract12} If $x \in \cB_{x_*}^{1}$ then
	\begin{equation}\label{eq:Bkplus1contract12}
	\EE{k}{\norm{\mB_{k+1} - \mH_*^{-1}}_{F(\mH_k)}^2}
	\leq
	(1-\rho)
	\norm{\mB_k - \mH_*^{-1}}_{F(\mH_k)}^2
	+d \norm{x_k-x_*}_{*}^2  \frac{(2 - \norm{x_k-x_*}_{*})^2}{(1 - \norm{x_k-x_*}_{*})^4},
	\end{equation}
	where 
	\begin{equation}\label{eq:rhoappen}
		 \rho \eqdef \inf\limits_{x \in \R^d} \lambda_{\min}\left(\E{\mH_{x}^{1/2} \mS(\mS^\top \mH_{x} \mS)^{-1}\mS^\top \mH_{x}^{1/2}}\right).
\end{equation} 
Consequently, if $x \in \cB_{x_*}^{1/4}$, then
		\begin{equation}\label{eq:Bkplus1contract14}
	\EE{k}{\norm{\mB_{k+1} - \mH_*^{-1}}_{F(\mH_k)}^2}
	\leq
	(1-\rho) \norm{\mB_k - \mH_*^{-1}}_{F(\mH_k)}^2
	+10 d \norm{x_k-x_*}_{*}^2.
	\end{equation} 
\end{lemma}
\begin{proof}

	We define the following projection matrix:
	\begin{equation}
	\mZ_k = \mH_k^{1/2}\mS(\mS^\top \mH_k\mS)^{-1}\mS^\top \mH_k^{1/2}.
	\end{equation}
	Hence from~\eqref{eq:BFGS} we have that
	\begin{equation}
	\mB_{k+1} = \mH_k^{-1/2}
	\left(
	\mZ_k + (\mI - \mZ_k)\mH_k^{1/2} \mB_k \mH_k^{1/2}(\mI - \mZ_k)
	\right)
	\mH_k^{-1/2}.
	\end{equation}
	To abbreviate the calculations we define  the linear operator $P : \sym^d \rightarrow \sym^d$:
	\begin{equation}
	P(\mA) \eqdef (\mI - \mZ_k) \mA (\mI - \mZ_k).
	\end{equation}
	
	Using this notation we rewrite $\norm{\mB_{k+1} - \mH_*^{-1}}_{F(\mH_k)}^2$:
	\begin{eqnarray*}
	\norm{\mB_{k+1} - \mH_*^{-1}}_{F(\mH_k)}^2
	&=&
	\norm{\mH_k^{1/2}(\mB_{k+1} - \mH_*^{-1})\mH_k^{1/2}}_{F}^2\\
	&=&
	\norm{\mZ_k + P(\mH_k^{1/2} \mB_k \mH_k^{1/2}) - \mH_k^{1/2}\mH_*^{-1}\mH_k^{1/2}}_{F}^2\\
	&=&
	\norm{P(\mH_k^{1/2} (\mB_k - \mH_*^{-1}) \mH_k^{1/2}) + P(\mH_k^{1/2} \mH_*^{-1} \mH_k^{1/2}) + \mZ_k  - \mH_k^{1/2}\mH_*^{-1}\mH_k^{1/2}}_{F}^2.
	\end{eqnarray*}
	Now note that the following equality holds
	\begin{equation}\label{eq:ZeqIPI}
	\mZ_k = \mI - P(\mI) = \mH_k^{1/2} \mH_k^{-1} \mH_k^{1/2} - P(\mH_k^{1/2} \mH_k^{-1} \mH_k^{1/2}).
	\end{equation}
	Using this equality and the shorthand
	\[\mR_k \eqdef \mH_k^{1/2}(  \mB_k-\mH_*^{-1})\mH_k^{1/2} \in \sym^d \qquad \mbox{and} \qquad
	\mD_k \eqdef \mH_k^{1/2}~(\mH_*^{-1} - \mH_k^{-1})\mH_k^{1/2} \in \sym^d,\]	
	we have that
	\begin{align}
	\norm{\mB_{k+1} - \mH_*^{-1}}_{F(\mH_k)}^2
	&=
	\norm{P(\mH_k^{1/2} (\mB_k - \mH_*^{-1}) \mH_k^{1/2}) + P(\mH_k^{1/2}( \mH_*^{-1} - \mH_k^{-1}) \mH_k^{1/2}) - \mH_k^{1/2}(\mH_*^{-1} - \mH_k^{-1})\mH_k^{1/2}}_{F}^2 \nonumber\\
	&=
	\underbrace{
		\norm{P(\mR_k)}_F^2
	}_{(I)}
	+
	\underbrace{
		\norm{P(\mD_k) - \mD_k}_{F}^2
	}_{(II)}+
	\underbrace{
		2\dotprod{P(\mR_k), P(\mD_k) - \mD_k}
	}_{(III)}.\label{eq:I-II-III}
	\end{align}
	
	One can show that $(III) = 0$. Note that $(I)$ has the following upper-bound:
	\begin{eqnarray}
	\norm{P(\mR_k)}_F^2 & 	= &
	\trace{\mR_k(\mI - \mZ_k)\mR_k(\mI - \mZ_k)}\nonumber\\
	&=&
	\trace{\mR_k(\mI - \mZ_k)\mR_k} 	-
	\trace{\mZ_k\mR_k(\mI - \mZ_k)\mR_k \mZ_k}\nonumber\\
	& \leq &
	\trace{\mR_k(\mI - \mZ_k)\mR_k}. \label{eq:Ibnd2323}
	\end{eqnarray}
	Taking expectation conditioned on $x_k$ now gives
		\begin{eqnarray}
	\EE{k}{\norm{P(\mR_k)}_F^2 }& \leq & \trace{\mR_k(\mI -\E{\mZ_k})\mR_k} \nonumber\\
	& \leq & (1- \lambda_{\min}(\E{\mZ_k})) 	\norm{\mB_k - \mH_*^{-1}}_{F(\mH_k)}^2 \\
	& \overset{\eqref{eq:rhoappen}}{\leq} & (1-\rho) 	\norm{\mB_k - \mH_*^{-1}}_{F(\mH_k)}^2,\label{eq:Ibnd}
	\end{eqnarray}
	where we used that $\trace{\mA\mB} \geq \lambda_{\min}(\mA) \trace{\mB}$ for any symmetric positive semi-definite matrices $\mA$ and $\mB.$
	Furthermore, $(II)$ has the following upper-bound:
	\begin{eqnarray}
	\norm{P(\mD_k) - \mD_k}_{F}^2
	&= &
	\trace{
		\left[
		\mD_k
		-
		(\mI - \mZ_k)
		\mD_k
		(\mI - \mZ_k)
		\right]^2
	}\nonumber \\
	&=&
	\norm{\mD_k}_F^2-
	\trace{
		(\mI - \mZ_k)
		\mD_k(\mI - \mZ_k)
		\mD_k(\mI - \mZ_k)
	}\nonumber \\
	&\leq &
 	\normfro{\mH_k - \mH_*}{\mH_*^{-1}}^2,
	\label{eq:needthismuchlater}
	\end{eqnarray}
	where we used that $	\trace{
		(\mI - \mZ_k)
		\mD_k(\mI - \mZ_k)
		\mD_k(\mI - \mZ_k)
	} \geq 0$ since the matrix within this trace is symmetric positive semi-definite.
	Using the above together with~\eqref{eq:Ibnd} and~\eqref{eq:I-II-III} gives
	\begin{equation}\label{eq:s9dfhdhsfs}
	\EE{k}{\norm{\mB_{k+1} - \mH_*^{-1}}_{F(\mH_k)}^2}
	\leq
	(1-\rho)
	\norm{\mB_k - \mH_*^{-1}}_{F(\mH_k)}^2
	+\normfro{\mH_k - \mH_*}{\mH_*^{-1}}^2.
	\end{equation}

	Now we need to bound the term $\normfro{\mH_k - \mH_*}{\mH_*^{-1}}$ using the properties of self-concordant functions. For $x_k \in \cB_{*}(x_*,1)$, we have 
	\begin{eqnarray*}
	 \normfro{\mH_k - \mH_*}{\mH_*^{-1}}^2 &= & \norm{\mH_*^{-1/2}(\mH_k - \mH_*)\mH_*^{-1/2}}_F^2 \nonumber  \\
	 &\leq & d \norm{\mH_*^{-1/2}(\mH_k - \mH_*)\mH_*^{-1/2}}_2^2 \nonumber  \\
	 & = &   d \norm{\mH_*^{-1}\mH_k - \mI}_{*}^2   \nonumber \\
	 & \overset{\eqref{eq:localhessdiff}}{\leq} & d \left(\frac{1}{\left(1 - \norm{x_k-x_*}_{*}\right)^2}-1\right)^2  \\
	 & = & d \norm{x_k-x_*}_{*}^2 \left(\frac{2 - \norm{x_k-x_*}_{*}}{(1 - \norm{x_k-x_*}_{*})^2}\right)^2.
	\end{eqnarray*}

	For $x \in \cB_{x_*}^{1/4}$ the final bound~\eqref{eq:Bkplus1contract14} follows by using that
	$$\frac{(2 - \norm{x_k-x_*}_{*})^2}{(1 - \norm{x_k-x_*}_{*})^4} \leq  \max_{\frac{1}{4} \geq t\geq 0} \frac{(2 - t)^2}{(1 - t)^4} = 10. $$
\end{proof}

\subsection{Detailed proof of Theorem~\ref{theo:sc}}


For convenience we repeat the statement of the theorem here.
\begin{theorem} \label{theo:scap}
Let
	\begin{equation}\label{eq:rho2ap} \rho \eqdef \inf\limits_{x \in \R^d} \lambda_{\min}\left(\E{\mH_{x}^{1/2}\mS(\mS^\top  \mH_x \mS)^{-1}\mS^\top \mH_{x}^{1/2}}\right).
	\end{equation}
 Consider the Lyapunov function
\begin{equation}\label{eq:lyapunov}
\Phi_{\sigma}^k  \eqdef  \sigma\norm{\mB_{k} - \mH_*^{-1}}_{F(\mH_*)}^2 +   \norm{x_k -x_*}_*,
\end{equation}
where $\sigma = \frac{3}{ \rho}$.
If $f$ is self-concordant and 
\begin{equation}\label{eq:phiboundinitial}\Phi_{\sigma}^0 \leq  \frac{1}{2}\min \left \{ \frac{3}{2} - \frac{1}{2}\sqrt{1+ 8\sqrt{\frac{1-\rho}{1-\frac{2}{3}\rho}}},  \;\;   \rho\frac{2-\rho}{69  d  +5 \rho}  
 \right\} ,
\end{equation}
then Algorithm~\ref{alg:BFGS} converges linearly according to
\begin{equation}\label{eq:Phirecurap}
\E{\Phi_{\sigma}^{k+1}  }\; \leq \; \left( 1- \frac{\rho}{2}\right)  \E{\Phi_{\sigma}^{k}}.
\end{equation}
Unrolling this recurrence, we get
$
\E{\norm{x_k -x_*}_*} \; \leq \;  \left( 1- \frac{\rho}{2}\right)^k \Phi_{\sigma}^0 .
$
\end{theorem}

\begin{proof}
We will proof~\eqref{eq:Phirecurap} by induction. 
Assuming 
\begin{equation}\label{eq:Phiinductproof}
\Phi_{\sigma}^k \leq  \frac{1}{2}\min \left \{ \frac{3}{2} - \frac{1}{2}\sqrt{1+ 8\sqrt{\frac{1-\rho}{1-\frac{2}{3}\rho}}},  \;\;   \rho\frac{2-\rho}{69  d  +5 \rho}  
 \right\} ,
\end{equation}  we will now show that it holds for $k+1$ and that~\eqref{eq:Phirecurap} holds. Note that 
due to~\eqref{eq:Phiinductproof} we have that 
\begin{equation}\label{eq:quarterbdxk}
 \norm{x_k -x_*}_* \leq \frac{1}{4}.
\end{equation}
  Indeed, since 
$$ \norm{x_k -x_*}_* \leq \frac{\rho}{2}  \frac{2-\rho}{69  d  +5 \rho} \; \implies \; \norm{x_k -x_*}_* \leq \frac{1}{4},$$
which holds because $0 \leq \rho \leq 1.$ 
Throughout the proof we will use many of the preceding intermediate results, many of which rely on~\eqref{eq:quarterbdxk}.

Taking expectation conditioned on $x_k$, using~\eqref{eq:normchangescy} and then Lemma~\ref{lem:Bkplus1contract12} gives
\begin{eqnarray*}
\EE{k}{\norm{\mB_{k+1} - \mH_*^{-1}}_{F(\mH_*)}^2 }&\leq & \left(\frac{1- \norm{x_k-x_*}_{*}}{1-2 \norm{x_k-x_*}_{*}}\right)^2\norm{\mB_{k+1} - \mH_*^{-1}}_{F(\mH_{k})}^2\nonumber \\
 & \leq &    \left(\frac{1- \norm{x_k-x_*}_{*}}{1-2 \norm{x_k-x_*}_{*}}\right)^2\left( (1-\rho)
	\norm{\mB_{k} - \mH_*^{-1}}_{F(\mH_{k})}^2 + 10 d  \norm{x_k-x_*}_{*}^2 \right) \nonumber\\
	& \leq &    \frac{1-\rho}{(1-2 \norm{x_k-x_*}_{*})^2}
	\norm{\mB_{k} - \mH_*^{-1}}_{F(\mH_{k})}^2  \\
	&& \quad +   10 d \norm{x_k-x_*}_{*}^2 \left(\frac{1- \norm{x_k-x_*}_{*}}{1-2 \norm{x_k-x_*}_{*}}\right)^2.
\end{eqnarray*}
Now using the change of norm bound~\eqref{eq:normchangescx} and that\footnote{This follows by taking the derivative
$$ \frac{d}{dt}\left(\frac{t-1}{2 t-1}\right)^2 = 2 \frac{t-1}{2 t-1} \frac{2t-1 - 2(t-1)}{(2t-1)^2}   =
2 \frac{t-1}{2 t-1}  \frac{1}{(2t-1)^2} >0$$ 
since $t < 1/4.$ Thus it is an increasing function and its maximum is at $t=1/4$. 
} 	
$$ x_k \in \cB_{*}(x_*,1/4) \quad \implies \quad  \left(\frac{1- \norm{x_k-x_*}_{*}}{1-2 \norm{x_k-x_*}_{*}}\right)^2 <2.25,$$
 in the above gives
\begin{eqnarray}
\EE{k}{\norm{\mB_{k+1} - \mH_*^{-1}}_{F(\mH_*)}^2 }
	& \overset{\eqref{eq:normchangescx} }{\leq} & \frac{1-\rho}{(1 - \norm{x_k-x_*}_{*})^2(1-2 \norm{x_k-x_*}_{*})^2}
	\norm{\mB_{k} - \mH_*^{-1}}_{F(\mH_{*})}^2 \nonumber \\
	&& \quad +   23 d \norm{x_k-x_*}_{*}^2.\label{eq:a983a8j38a3}
	\end{eqnarray}

Turning our attention to the Lyapunov function~\eqref{eq:lyapunov}
and combining~\eqref{eq:a983a8j38a3} with~\eqref{eq:iteratescbnd14} gives
\begin{eqnarray*}
\Phi_{\sigma}^{k+1} & \leq &   \sigma\left( \frac{1-\rho}{(1 - \norm{x_k-x_*}_{*})^2(1-2 \norm{x_k-x_*}_{*})^2} +\frac{1}{2\sigma}\right)
	\norm{\mB_{k} - \mH_*^{-1}}_{F(\mH_{*})}^2 + \left(23\sigma d  +5\right) \norm{x_k-x_*}_{*}^2 .
\end{eqnarray*}

Now we have that
\begin{equation}
\norm{x_k -x_*} \;\leq \; \frac{3}{4} - \frac{1}{4}\sqrt{1+ 8\sqrt{\frac{1-\rho}{1-\frac{2}{3}\rho}}} \quad \implies \quad \frac{1-\rho}{(1 - \norm{x_k-x_*}_{*})^2(1-2 \norm{x_k-x_*}_{*})^2} \leq 1-\frac{2\rho}{3},\label{eq:a39jao9j3a9j3}
\end{equation}
which we prove in Lemma~\ref{lem:iteratesbnd} further down.  Consequently, by imposing the constraint on the left hand side of~\eqref{eq:a39jao9j3a9j3} we have that
\begin{eqnarray}
\Phi_{\sigma}^{k+1} & \leq &   \sigma\left( 1-\frac{2\rho}{3}+\frac{1}{2\sigma}\right)
	\norm{\mB_{k} - \mH_*^{-1}}_{F(\mH_{*})}^2 + \left(23\sigma d  +5\right) \norm{x_k-x_*}_{*}^2 .\label{eq:lasttimerqndom}
\end{eqnarray}
Now choosing $\sigma = \frac{3}{ \rho}$ so that 
\[  \sigma\left( 1-\frac{2\rho}{3}+\frac{1}{2\sigma}\right)
	\norm{\mB_{k} - \mH_*^{-1}}_{F(\mH_{*})}^2 =  \sigma\left( 1-\frac{\rho}{2}\right)
	\norm{\mB_{k} - \mH_*^{-1}}_{F(\mH_{*})}^2,\]
and further restricting $\norm{x_k -x_*}$ so that
\begin{equation}
 \norm{x_k-x_*}_{*} \leq \frac{\rho}{2} \frac{2-\rho}{69  d  +5 \rho} \quad \implies \quad \left(23\sigma d  +5\right)\norm{x_k-x_*}_{*}^2 \leq \left(1-\frac{\rho}{2}\right)\norm{x_k-x_*}_{*},
\end{equation}
we have from~\eqref{eq:lasttimerqndom} that
\begin{eqnarray}
\Phi_{\sigma}^{k+1} & \leq &   \sigma\left( 1-\frac{\rho}{2}\right)
	\norm{\mB_{k} - \mH_*^{-1}}_{F(\mH_{*})}^2 +\left(1-\frac{\rho}{2}\right) \norm{x_k-x_*}_{*} 
	 \;= \; \left(1-\frac{\rho}{2}\right) \Phi_{\sigma}^{k}.  \nonumber
\end{eqnarray}
Finally, since $0 \leq \rho \leq 1$ we have that $ \Phi_{\sigma}^{k+1} \leq  \Phi_{\sigma}^{k}$ 
and thus~\eqref{eq:Phiinductproof} holds with $k+1$ in place of $k$,  which proves the induction.
%
%

\end{proof}

\begin{lemma} \label{lem:iteratesbnd}
For $t<\frac{1}{2}$ and $\rho <1$ we have that  
$$ t <   \frac{3}{4} - \frac{1}{4}\sqrt{1+ 8\sqrt{\frac{1-\rho}{1-\frac{2}{3}\rho}}}\; \implies \;    \frac{1-\rho}{(1 - t)^2(1-2 t)^2}  \leq 1-\frac{2\rho}{3}. $$
\end{lemma}

\begin{proof}
Using that $t<1/2$ and $\rho <1$ we have that 
 $$  \frac{1-\rho}{(1 - t)^2(1-2 t)^2}  \leq 1-\frac{2\rho}{3} \Leftrightarrow $$
 $$ 1-\rho \leq  (1 - t)^2(1-2 t)^2 \left(1-\frac{2\rho}{3}   \right)\Leftrightarrow$$ 
 \begin{equation}\label{eq:2nddegreet}
   \sqrt{\frac{1-\rho}{1-\frac{2}{3}\rho}} \leq  (1 - t)(1-2 t) . 
\end{equation}


 Let $\psi \eqdef \sqrt{\frac{1-\rho}{1-\frac{2}{3}\rho}} .$
The above holds to \emph{equality} if 
$$  1-\psi -3t +2t^2  = c +bt + a^2 t. $$
  $$\frac{-b \pm \sqrt{ b^2 -4ac}}{2a} = \frac{3 \pm \sqrt{ 9 -8(1-\psi)}}{4} .$$ 
 $$t =\frac{3}{4} - \frac{\sqrt{9-8(1-\psi)}}{4} = \frac{3}{4} - \frac{\sqrt{1+8\psi}}{4}  = \frac{3}{4} - \frac{1}{4}\sqrt{1+ 8\sqrt{\frac{1-\rho}{1-\frac{2}{3}\rho}}} <\frac{3}{4}. $$
Furthermore~\eqref{eq:2nddegreet} holds for all $t \leq \frac{3}{4} -  \frac{\sqrt{1+8\psi}}{4}  .$ This is because $p(t) \eqdef 1 -3t +2t^2$ is decreasing on the interval $t< \frac{3}{4}$ as can  be verified by taking the derivative
$p'(t)  = 4t -3$. 
\end{proof}

\section{Proof of Theorem~\ref{theo:locallinear}}

For this proof we use the Monotonic option in Algorithm~\ref{alg:BFGS}. Furthermore, for the proof we will use the notation
\[x_{+} = x_k - \mB_k\nabla f(x_k).\]

\subsection{Getting ready for the proof}

We first state and prove four lemmas.

\begin{lemma}

	Let $\mW \in \sym^d$.
	For all $x,y\in \R^d$ the following inequality holds:
	\begin{equation}\label{eq:normchange}
		\normfro{\mW}{\mH_{x}} \leq \left(1 + \frac{L_2}{\mu}\norm{x - y}_2\right)\normfro{\mW}{\mH_{y}}.
	\end{equation}
\end{lemma}

\begin{proof}
	\begin{eqnarray*}
		\normfro{\mW}{\mH_{x}}
		&\leq &
		\normfro{\mW}{\mH_{y}}\norm{\mH_{y}^{-1/2} \mH_{x} \mH_{y}^{-1/2}}_2\\
		&\leq &
		\normfro{\mW}{\mH_{y}}\left(1 + \norm{\mH_{y}^{-1/2} [\mH_{x} - \mH_{y}] \mH_{y}^{-1/2}}_2\right)\\
		&\leq &
		\normfro{\mW}{\mH_{y}}\left(1 + \norm{\mH_{y}^{-1}}_2\norm{\mH_{x} - \mH_{y}}_2\right)\\
		&\leq &
		\normfro{\mW}{\mH_{y}}\left(1 + \mu^{-1}L_2\norm{x - y}_2\right).
	\end{eqnarray*}
\end{proof}

\begin{lemma}

Let $\mH_* \eqdef \nabla^2 f(x_*).$ It follows that \label{lem:Hessstacontract}
	\begin{align*}
	\norm{x_k - x_* - \mH_*^{-1}\nabla f(x_k)}_2 
	& \leq \frac{L_2}{2\mu} \norm{x_k -x_*}^2_2.
	\end{align*}
\end{lemma}
\begin{proof}
	By the fundamental theorem of calculus,
	\begin{equation}\label{eq:bi87fg8d} \nabla f(x_k) = \nabla f(x_k)-\nabla f(x_*) =\int_{0}^1\nabla^2 f(x_* + t(x_k-x_*))(x_k -x_*) dt.\end{equation}
	The result now follows from
	\begin{eqnarray*}
	\norm{x_k-x_*-\mH_*^{-1} \nabla f(x_k)}_2 &\overset{\eqref{eq:bi87fg8d}}{=} &
	\norm{\mH_*^{-1} \left(\mH_*-\int_{0}^1\nabla^2 f(x_* + t(x_k-x_*) ) dt\right)(x_k -x_*)}_2 \\
	&\leq & \norm{\mH_*^{-1}}_2 \norm{\mH_*-\int_{0}^1 \nabla^2 f(x_* + t(x_k-x_*)) dt}_2  \norm{x_k -x_*}_2 \\
	&=& \norm{\mH_*^{-1}}_2 \norm{\int_{0}^1  \left[\mH_*-\nabla^2 f(x_* + t(x_k-x_*)) \right]dt}_2  \norm{x_k -x_*}_2 \\
	& \leq & \norm{\mH_*^{-1}}_2 \left[\int_{0}^1\norm{\mH_*-\nabla^2 f(x_* + t(x_k-x_*)) }_2dt \right] \norm{x_k -x_*}_2  \\
	& \overset{\eqref{eq:HessLip}+\eqref{eq:strconv}}{\leq} & \frac{L_2}{\mu}  \left(\int_{0}^1\norm{t(x_k-x_*)}_2 dt \right)\norm{x_k -x_*}_2  \\
	& = & \frac{L_2}{2\mu}  \norm{x_k -x_*}_2^2.
	\end{eqnarray*}
\end{proof}

\begin{lemma}\label{lem:sqrtfbnd}
	\begin{equation}\label{eq:sqrtfbnd}
	\sqrt{f(x_+) - f(x_*)}
	\leq
	\frac{\sqrt{2L_1} L_2}{\mu^2}  \left(f(x_k) - f(x_*)\right)
	+
	\frac{ L_1^{5/2}}{2\sqrt{2}\mu L_2}\norm{ \mB_k-\mH_*^{-1} }_{F(\mH_k)}^2.
	\end{equation}
\end{lemma}
\begin{proof}
	We start with an upper-bound for $\norm{x_+ - x_*}_2$:
	\begin{eqnarray*}	
	\norm{x_+ - x_*}_2
	&=&
	\norm{x_k - x_* - \mB_k\nabla f(x_k)}_2\\
	&=&
	\norm{x_k - x_* - \mH_*^{-1}\nabla f(x_k)  + (\mH_*^{-1} - \mB_k)\nabla f(x_k)}_2 \\
	&\leq&
	\norm{x_k - x_* - \mH_*^{-1}\nabla f(x_k)}_2
	+
	\norm{(\mH_*^{-1} - \mB_k)\nabla f(x_k)}_2\\
	&\overset{\text{Lemma}~\ref{lem:Hessstacontract} +\eqref{eq:GradLip} }{\leq}&
	\frac{L_2}{2\mu}\norm{x_k - x_*}_2^2
	+
	L_1\norm{\mH_*^{-1} - \mB_k}_2\norm{x_k - x_*}_2.
	\end{eqnarray*}	
	Now using that $ab \leq \frac{c}{2} a^2 + \frac{1}{2c} b^2$ for every $a,b,c>0$ with $a = \norm{x_k - x_*}_2,$	$b = \norm{\mH_*^{-1} - \mB_k}_2$ and $c = \frac{L_2}{\mu L_1}$ in the above gives
	\begin{eqnarray*}
	\norm{x_+ - x_*}_2&\leq &
	\frac{L_2}{\mu}\norm{x_k - x_*}_2^2
	+
	\frac{\mu L_1^2}{2L_2}\norm{\mH_*^{-1} - \mB_k}_2^2\\
	&\leq &
	\frac{L_2}{\mu}\norm{x_k - x_*}_2^2
	+
	\frac{\mu L_1^2}{2L_2}\norm{\mH_*^{-1} - \mB_k}_{F(\mH_k)}^2\norm{\mH_k^{-1}}_2^2\\
	&\leq &
	\frac{L_2}{\mu}\norm{x_k - x_*}_2^2
	+
	\frac{ L_1^2}{2\mu L_2}\norm{\mH_*^{-1} - \mB_k}_{F(\mH_k)}^2.
	\end{eqnarray*}
	Using smoothness~\eqref{eq:GradLip}, the above inequalitym and then strong convexity~\eqref{eq:strconv} in that order, we have that
	\begin{eqnarray*}
	\sqrt{f(x_+) - f(x_*)}
	& \leq &
	\sqrt{\frac{L_1}{2}} \norm{x_+ - x_*}_2 \\
	 &\leq &
	\sqrt{\frac{L_1}{2}} \frac{L_2}{\mu}\norm{x_k - x_*}_2^2
	+
	\frac{ L_1^{5/2}}{2\sqrt{2}\mu L_2}\norm{\mH_*^{-1} - \mB_k}_{F(\mH_k)}^2\\
	& \leq &
	\frac{\sqrt{2L_1} L_2}{\mu^2}  \left(f(x_k) - f(x_*)\right)
	+
	\frac{ L_1^{5/2}}{2\sqrt{2}\mu L_2}\norm{\mH_*^{-1} - \mB_k}_{F(\mH_k)}^2.
	\end{eqnarray*}
\end{proof}

\begin{lemma} \label{lem:Bkplus1contract1}
	\begin{equation}\label{eq:Bkplus1contract1}
\EE{k}{\norm{\mB_{k+1} - \mH_*^{-1}}_{F(\mH_k)}^2}
	\leq
	(1-\rho)
	\norm{\mB_k - \mH_*^{-1}}_{F(\mH_k)}^2
	+
	\frac{2dL_2^2}{\mu^3} (f(x_k) - f(x_*)).
	\end{equation}
\end{lemma}

\begin{proof}
	Recalling~\eqref{eq:s9dfhdhsfs} we have that
	
		\begin{equation}\label{eq:s9dfhdhsfs22}
	\EE{k}{\norm{\mB_{k+1} - \mH_*^{-1}}_{F(\mH_k)}^2 }
	\leq
	(1-\rho)
	\norm{\mB_k - \mH_*^{-1}}_{F(\mH_k)}^2
	+\normfro{\mH_k - \mH_*}{\mH_*^{-1}}^2.
	\end{equation}	
Now note that
	\begin{eqnarray}
	\norm{\mH_*^{-1/2} (\mH_k - \mH_*)\mH_*^{-1/2}}_F^2
	& \leq &
	d \norm{\mH_*^{-1}}_2^2 \norm{\mH_k - \mH_*}_2^2 \nonumber \\
	&\overset{\eqref{eq:strconv} + \eqref{eq:HessLip} }{\leq} &
	\frac{dL_2^2}{\mu^2} \norm{x_k - x_*}_2^2. \label{eq:IIbnd}
	\end{eqnarray}

	Finally, taking expectation over~\eqref{eq:s9dfhdhsfs22} and using the above gives
	\begin{eqnarray*}
	\EE{k}{\norm{\mB_{k+1} - \mH_*^{-1}}_{F(\mH_k)}^2}
	& \leq &
	(1-\rho)
	\norm{\mB_k - \mH_*^{-1}}_{F(\mH_k)}^2
	+
	\frac{dL_2^2}{\mu^2} \norm{x_k - x_*}_2^2\\
	&\leq &
	(1-\rho)
	\norm{\mB_k - \mH_*^{-1}}_{F(\mH_k)}^2
	+
	\frac{2dL_2^2}{\mu^3} (f(x_k) - f(x_*)).
	\end{eqnarray*}
\end{proof}

Now, we consider the following Lyapunov function:
\begin{equation}
\Psi_k = \sqrt{f(x_k) - f(x_*)} + \beta \norm{\mB_k - \mH_*^{-1}}_{F(\mH_*)}^2,
\end{equation}
where  $\beta = \frac{4\sqrt{2}L_1^{5/2}}{\mu L_2 \rho}$.

\subsection{The Proof}

Having established several key lemmas, we are now ready to prove Theorem~\ref{theo:locallinear}.

\begin{proof}
	\begin{eqnarray*}
		\E{\Psi_{k+1}}
		&= &
		\sqrt{f(x_{k+1}) - f(x_*)} + \beta \E{\norm{\mB_{k+1} - \mH_*^{-1}}_{F(\mH_*)}^2}\\
		&\leq &
		\sqrt{f(x_+) - f(x_*)} + \beta \E{\norm{\mB_{k+1} - \mH_*^{-1}}_{F(\mH_*)}^2}\\
		&\overset{\eqref{eq:sqrtfbnd}+\eqref{eq:normchange}}{\leq} &
		\frac{\sqrt{2L_1}L_2}{\mu^2} (f(x_k) - f(x_*))
		+
		\frac{ L_1^{5/2}}{2\sqrt{2}\mu L_2}\norm{ \mB_k-\mH_*^{-1} }_{F(\mH_k)}^2 \\
		&& \qquad +
		\beta\left[1 + \frac{L_2}{\mu}\norm{x_k - x_*}_2\right]^2 \E{\norm{\mB_{k+1} - \mH_*^{-1}}_{F(\mH_k)}^2} \\
		&\overset{\eqref{eq:Bkplus1contract1}}{\leq} &
		\frac{\sqrt{2L_1}L_2}{\mu^2} (f(x_k) - f(x_*))
		+
		\frac{ L_1^{5/2}}{2\sqrt{2}\mu L_2}\norm{ \mB_k-\mH_*^{-1} }_{F(\mH_k)}^2 \\
		&& \qquad +
		\beta\left[1 + \frac{L_2}{\mu}\norm{x_k - x_*}_2\right]^2 \left[
			(1-\rho)
			\norm{\mB_k - \mH_*^{-1}}_{F(\mH_k)}^2
			+
			\frac{2dL_2^2}{\mu^3} (f(x_k) - f(x_*))
		\right] \\
		&\leq &
		\left[
			\frac{\sqrt{2L_1}L_2}{\mu^2}  + \frac{2\beta dL_2^2}{\mu^3}\left[1 + \frac{L_2}{\mu}\norm{x_k - x_*}_2\right]^2
		\right](f(x_k) - f(x_*)) \\
		&& \qquad +
		\left[
			(1-\rho)\left[1 + \frac{L_2}{\mu}\norm{x_k - x_*}_2\right]^4 
			+
			\frac{ L_1^{5/2}}{2\beta\sqrt{2}\mu L_2}\left[1 + \frac{L_2}{\mu}\norm{x_k - x_*}_2\right]^2
		\right]\beta\norm{\mB_k - \mH_*^{-1}}_{F(\mH_*)}^2.
	\end{eqnarray*}
	Now, since $f(x_k) -f(x_*)$ is non-increasing by line~\ref{ln:monotonic} in Algorithm~\ref{alg:BFGS}, we have that $\norm{x_k - x_*}_2 \leq \sqrt{\frac{2}{\mu}(f(x_k) - f(x_*))} \leq  \sqrt{\frac{2}{\mu}(f(x_0) - f(x_*))} \leq \sqrt{\frac{2\cF}{\mu}}$:
	\begin{eqnarray*}
	\E{\Psi_{k+1}}
	&\leq &
	\left[
	\frac{\sqrt{2L_1}L_2}{\mu^2}  + \frac{2\beta dL_2^2}{\mu^3}\left[1 + \frac{\sqrt{2\cF}L_2}{\mu^{3/2}}\right]^2
	\right](f(x_k) - f(x_*))\\
	&& \qquad +
	\left[
	(1-\rho)\left[1 + \frac{\sqrt{2\cF}L_2}{\mu^{3/2}}\right]^4
	+
	\frac{ L_1^{5/2}}{2\beta\sqrt{2}\mu L_2}\left[1 + \frac{\sqrt{2\cF}L_2}{\mu^{3/2}}\right]^2
	\right]\beta\norm{\mB_k - \mH_*^{-1}}_{F(\mH_*)}^2.
	\end{eqnarray*}
	Equation \eqref{eq:1} implies  $\frac{\sqrt{2\cF}L_2}{\mu^{3/2}} \leq 1$ and we get
	\begin{eqnarray*}
	\E{\Psi_{k+1}}
	&\leq &
	\left[
	\frac{\sqrt{2L_1}L_2}{\mu^2}  + \frac{8\beta dL_2^2}{\mu^3}
	\right](f(x_k) - f(x_*)) \\
	&& \quad +
	\left[
	(1-\rho)\left[1 + \frac{15\sqrt{2\cF}L_2}{\mu^{3/2}}\right]
	+
	\frac{ \sqrt{2}L_1^{5/2}}{\beta\mu L_2}
	\right]\beta\norm{\mB_k - \mH_*^{-1}}_{F(\mH_*)}^2\\
	&\leq &
	\left[
	\frac{\sqrt{2L_1}L_2}{\mu^2}  + \frac{8\beta dL_2^2}{\mu^3}
	\right](f(x_k) - f(x_*)) \\	
	&& \quad +
	\left[
	1-\rho +  \frac{15\sqrt{2\cF} L_2}{\mu^{3/2}}
	+
	\frac{ \sqrt{2}L_1^{5/2}}{\beta\mu L_2}
	\right]\beta\norm{\mB_k - \mH_*^{-1}}_{F(\mH_*)}^2.
	\end{eqnarray*}
	Using $\beta = \frac{4\sqrt{2}L_1^{5/2}}{\mu L_2 \rho}$ and using $\frac{15\sqrt{2\cF} L_2}{\mu^{3/2}} \leq\frac{\rho}{4}$, which follows from \eqref{eq:1}, we get
	\begin{align*}
	\E{\Psi_{k+1}}
	\leq
	\left[
	\frac{\sqrt{2L_1}L_2}{\mu^2}  + \frac{32\sqrt{2}dL_1^{5/2}L_2}{\rho\mu^4}
	\right]\sqrt{\cF\cdot (f(x_k) - f(x_*))}+
	\left(1- \frac{\rho}{2}\right)\beta\norm{\mB_k - \mH_*^{-1}}_{F(\mH_*)}^2.
	\end{align*}
	Finally, using  \eqref{eq:1} we get
	\begin{equation}
	\E{\Psi_{k+1}}
	\leq
	\frac{1}{2}\sqrt{(f(x_k) - f(x_*))}
	+
	\left(1- \frac{\rho}{2}\right)\beta\norm{\mB_k - \mH_*^{-1}}_{F(\mH_*)}^2 \leq
	\left(1- \frac{\rho}{2}\right)\Psi_k.
	\end{equation}

\end{proof}

\section{Proof of Theorem~\ref{thm:superlinear} }

\begin{lemma}\label{lem:1}
	$\sqrt{\frac{f(x_{k+1}) - f(x_*)}{f(x_k) - f(x_*)}}$ converges to 0 with probability 1.
\end{lemma}

\begin{proof}
	We start with an upper-bound for $\sqrt{f(x_{k+1}) - f(x_*)}$:
	\begin{eqnarray*}
	\sqrt{f(x_{k+1}) - f(x_*)}
	&\leq &
	\sqrt{f(x_{+}) - f(x_*)} \\
	& \leq &
	\sqrt{\frac{L_1}{2}}\norm{x_+ - x_*}_2\\
	&= &
	\sqrt{\frac{L_1}{2}}\norm{x_k - x_* - \mB_k\nabla f(x_k)}_2\\
	&= &
	\sqrt{\frac{L_1}{2}}\norm{x_k - x_* - \mH_*^{-1}\nabla f(x_k)  + (\mH_*^{-1} - \mB_k)\nabla f(x_k)}_2 \\
	&\leq &
	\sqrt{\frac{L_1}{2}}\norm{x_k - x_* - \mH_*^{-1}\nabla f(x_k)}_2
	+
	\sqrt{\frac{L_1}{2}}\norm{(\mH_*^{-1} - \mB_k)\nabla f(x_k)}_2\\
	&\overset{\text{Lemma}~\ref{lem:Hessstacontract} +\eqref{eq:GradLip}}{\leq} &
	\frac{L_2\sqrt{L_1}}{2\sqrt{2}\mu}\norm{x_k - x_*}_2^2
	+
	\frac{L_1^{3/2}}{\sqrt{2}}\norm{\mH_*^{-1} - \mB_k}_2\norm{x_k - x_*}_2 \\
	&\overset{\eqref{eq:strconv}}{\leq} &
	\frac{L_2\sqrt{L_1}}{\sqrt{2}\mu^2}(f(x_k) - f(x_*))
	+
	\frac{L_1^{3/2}}{\sqrt{\mu}}\norm{\mH_*^{-1} - \mB_k}_2\sqrt{f(x_k) - f(x_*)} \\
	&\leq &
	\frac{L_2\sqrt{L_1}}{\sqrt{2}\mu^2}(f(x_k) - f(x_*))
	+
	\frac{L_1^{3/2}}{\mu^{3/2}}\norm{\mH_*^{-1} - \mB_k}_{F(\mH_{*})}\sqrt{f(x_k) - f(x_*)}.
	\end{eqnarray*}
	This leads to
	\begin{eqnarray*}
	\sqrt{\frac{f(x_{k+1}) - f(x_*)}{f(x_k) - f(x_*)}}
	&\leq &
	\frac{L_2\sqrt{L_1}}{\sqrt{2}\mu^2}\sqrt{f(x_k) - f(x_*)}
	+
	\frac{L_1^{3/2}}{\mu^{3/2}}\norm{\mH_*^{-1} - \mB_k}_{F(\mH_{*})}\\
	&\leq &
	a \Psi_k + b\sqrt{\Psi_k},
	\end{eqnarray*}
	where $a = \frac{L_2\sqrt{L_1}}{\sqrt{2}\mu^2}$ and $b = \frac{L_1^{3/2}}{\mu^{3/2}\sqrt{\beta}}$.
	After taking expectation, we get
	\begin{eqnarray*}
	\E{\sqrt{\frac{f(x_{k+1}) - f(x_*)}{f(x_k) - f(x_*)}}}
	&\leq &
	\E{a \Psi_k + b\sqrt{\Psi_k}} \\
	& \leq &
	a\E{\Psi_k}+b\sqrt{\E{\Psi_k}}\\
	&\leq &
	a(1-\rho)^k\Psi_0 + b\sqrt{(1-\rho)^k\Psi_0}\\
	&\leq & cq^k,
	\end{eqnarray*}
	where $q = \sqrt{1-\rho}\in (0,1)$ and $c = a \Psi_0 + b\sqrt{\Psi_0}$.
	Now, we choose arbitrary $\epsilon$ and apply Markov's inequality:
	\begin{align*}
	\prob{\sqrt{\frac{f(x_{k+1}) - f(x_*)}{f(x_k) - f(x_*)}} \geq  \epsilon} \leq \frac{1}{\epsilon} \E{\sqrt{\frac{f(x_{k+1}) - f(x_*)}{f(x_k) - f(x_*)}}}
	\leq \frac{cq^k}{\epsilon}.
	\end{align*}
	Hence, for all $\epsilon > 0$
	\begin{equation}
	\sum_{k=0}^{\infty}\prob{\sqrt{\frac{f(x_{k+1}) - f(x_*)}{f(x_k) - f(x_*)}} \geq  \epsilon} < \infty,
	\end{equation}
	which means that $\sqrt{\frac{f(x_{k+1}) - f(x_*)}{f(x_k) - f(x_*)}}$ converges to $0$ in probability sufficiently quickly and hence it converges to $0$ almost surely.
	
\end{proof}

\begin{lemma}\label{lem:2}
	$\sqrt{f(x_k) - f(x_*)}$ converges to 0 with probability 1.
\end{lemma}
\begin{proof}
	$\E{\sqrt{f(x_k) - f(x_*)}}$ converges linearly to 0, thus it converges in probability to 0 sufficiently quickly. Hence, $\sqrt{f(x_k) - f(x_*)}$ converges to 0 almost surely.
\end{proof}



\section{Proof of Theorem~\ref{prop:svd_sketch_log}}
\begin{proof}
The proof follows verbatim the proof of Theorem~\ref{prop:svd_sketch}.
\end{proof}

\end{document}